\documentclass[letterpaper,12pt]{amsart}

\usepackage{latexsym,array,delarray,amsthm,amssymb,epsfig,color,amsmath}
\usepackage{graphics,graphicx, caption, comment}
\usepackage{textcomp}
\usepackage{MnSymbol}
\usepackage{romanbar}
\usepackage{caption}
\usepackage{subcaption}

\usepackage[normalem]{ulem}

\theoremstyle{plain}
\newtheorem{thm}{Theorem}[section]
\newtheorem{lemma}[thm]{Lemma}
\newtheorem{prop}[thm]{Proposition}
\newtheorem{cor}[thm]{Corollary}

\newtheorem*{thm*}{Theorem 1}
\newtheorem*{lemma*}{Lemma}
\newtheorem*{prop*}{Proposition}
\newtheorem*{cor*}{Corollary}
\newtheorem*{conj*}{Conjecture}
\newtheorem*{ques*}{Question}
\newtheorem*{challenge*}{Challenge}
\newtheorem{remark}[thm]{Remark}

\theoremstyle{definition}

\theoremstyle{remark}

\newcommand{\h}{ {\rm h}} 

\newcommand{\I}{\normalfont{\Romanbar{1}}} 
\newcommand{\II}{\normalfont{\Romanbar{2}}} 

\usepackage{layout}
\usepackage{multicol}
\usepackage{color}

\def\restriction#1#2{\mathchoice
	{\setbox1\hbox{${\displaystyle #1}_{\scriptstyle #2}$}
		\restrictionaux{#1}{#2}}
	{\setbox1\hbox{${\textstyle #1}_{\scriptstyle #2}$}
		\restrictionaux{#1}{#2}}
	{\setbox1\hbox{${\scriptstyle #1}_{\scriptscriptstyle #2}$}
		\restrictionaux{#1}{#2}}
	{\setbox1\hbox{${\scriptscriptstyle #1}_{\scriptscriptstyle #2}$}
		\restrictionaux{#1}{#2}}}
\def\restrictionaux#1#2{{#1\,\smash{\vrule height .8\ht1 depth .85\dp1}}_{\,#2}} 

\title[FEMs for Thin Structures with Folding]{Finite Element Methods for the Stretching and Bending of Thin Structures with Folding}

\author{Andrea Bonito}
\address[Andrea Bonito]{Department of Mathematics, Texas A\&M University, College Station, TX 77845, USA.}
\email{bonito@math.tamu.edu}
\thanks{AB and AM are partially supported by the NSF Grant DMS-2110811. DG is partially supported by the NSERC Grant RGPIN-2021-04311.}
\author{Diane Guignard}
\address[Diane Guignard]{Department of Mathematics and Statistics
	\\ University of Ottawa,
	Ottawa, ON K1N 6N5, Canada.}
\email{dguignar@uottawa.ca}
\author{Angelique Morvant}
\address[Angelique Morvant]{Department of Mathematics, Texas A\&M University, College Station, TX 77845, USA.}
\email{mae4102@tamu.edu}
\date{\today}

\begin{document}

\maketitle

\begin{abstract}
 In \cite{Prestrain_BGNY_2022}, a local discontinuous Galerkin method was proposed for approximating the large bending of prestrained plates, and in \cite{Prestrain_theoretical_BGNY} the numerical properties of this method were explored. These works considered deformations driven predominantly by bending. Thus, a bending energy with a metric constraint was considered. We extend these results to the case of an energy with both a bending component and a nonconvex stretching component, and we also consider folding across a crease. The proposed discretization of this energy features a continuous finite element space, as well as the discrete Hessian used in \cite{Prestrain_BGNY_2022, Prestrain_theoretical_BGNY}. We establish the $\Gamma$-convergence of the discrete to the continuous energy and also present an energy-decreasing gradient flow for finding critical points of the discrete energy. Finally, we provide numerical simulations illustrating the convergence of minimizers and the capabilities of the model.
\end{abstract}
	

\maketitle


\section{Introduction}

The deformation of thin materials occurs widely in nature, from the snapping of the venus flytrap \cite{Flytrap_FSDM,Speck2020} to the growth of leaves and flowers \cite{Wrinkling_BK_2014, Review_LM_2021}. It may also occur in man-made applications, such as the bending of hydrogel discs \cite{KVS_2011, Review_LM_2021} and the movement of microswimmers through the body \cite{Micro_NYA_2015}. In all of these examples, the bending of the material is the result of a prestrain; that is, internal stresses in the flat configuration cause the plate to deform into some more complicated shape. This causes (potentially) large deformations that result from relatively small actuations. 

The deformation of prestrained materials has been studied in several works; see for instance \cite{Wrinkling_BK_2014, Prestrain_BGNY_2022, Prestrain_theoretical_BGNY, NonEuclidPlates_ESK_2009}. Efrati, Sharon, and Kupferman \cite{ NonEuclidPlates_ESK_2009} use Kirchhoff-Love assumptions to derive a 2D energy for prestrained plates with a stretching and a bending component.  In \cite{Prestrain_BGNY_2022}, the energy is a bending energy with stretching enforced by a nonlinear and nonconvex constraint involving the prestrain metric. The reference \cite{BGM_2022} summarizes these and other models for the bending of prestrained plates, including extensions to plates with a bilayer and plates with folding. We refer to \cite{BNY2022a,BNY2022b} for some work where stretching is responsible for the deformation of membranes.

The present work builds on that of \cite{Prestrain_theoretical_BGNY}, in which bending is the main mechanism for the deformation of plates. In \cite{Prestrain_theoretical_BGNY}, the authors seek to solve the following constrained minimization problem:
\begin{equation}
	\min_{{\bf y} \in \tilde{\mathbb{A}}} E^B({\bf y}),
\end{equation}
where 
\begin{equation} \label{eq:bend_energy}
	E^B({\bf y}) := \frac{1}{24} \sum_{m=1}^3 \int_{\Omega}\left(2\mu|{\bf g}^{-1/2}D^2 y_m {\bf g}^{-1/2}|^2 + \frac{2\mu\lambda}{2 \mu+\lambda}\text{tr}\Big({\bf g}^{-1/2}D^2y_m{\bf g}^{-1/2}\Big)^2\right)
\end{equation}
is the \emph{bending energy}. Here, ${\bf y} = (y_m)_{m=1}^3:\Omega\rightarrow\mathbb{R}^3$ denotes a deformation of the midplane $\Omega$ of the plate, ${\bf g}:\Omega\rightarrow\mathbb{R}^{2 \times 2}$ is symmetric, positive definite, and corresponds to the so-called target metric, and $\mu, \lambda$ are the first and second Lam{\'e} coefficients of the material. The admissible set $\tilde{\mathbb{A}}$ is given by
\begin{equation}
	\tilde{\mathbb{A}} := \{{\bf y} \in [H^2(\Omega)]^3 \ | \ \nabla {\bf y}^T \nabla {\bf y} = {\bf g} \text{ a.e. in } \Omega \}.
\end{equation}
The bending energy was derived rigorously using $\Gamma$-convergence in \cite{Bending2_FJM_2002} for the case $g=I_2$ (isometry constraint) and in \cite{lewicka2011} for a general metric $g$. 

We extend the results of \cite{Prestrain_theoretical_BGNY} to the case where both bending and stretching drive the deformation of the plate. Moreover, as in \cite{BBH2021, bartels2022error, BNY2023, BGM_2022}, we allow for folding of the plate along a curve $\Sigma \subset\Omega$. In this case, the energy of the plate takes the form
\begin{equation} \label{eq:bend_stretch_energy}
	E({\bf y}) = E^S({\bf y}) + \theta^2 E^B({\bf y}),
\end{equation}
where
\begin{equation}
	E^S({\bf y}) := \frac{1}{8}\int_{\Omega} \Big( 2\mu |{\bf g}^{-1/2}(\nabla {\bf y}^T \nabla {\bf y}-{\bf g}){\bf g}^{-1/2}|^2 + \lambda \text{tr}({\bf g}^{-1/2}(\nabla {\bf y}^T \nabla {\bf y} - {\bf g}){\bf g}^{-1/2})^2 \Big) 
\end{equation}
is the \emph{stretching energy}, $E^B({\bf y})$ is given by \eqref{eq:bend_energy} with $\Omega$ replaced by $\Omega \setminus \Sigma$, and $\theta>0$ is the thickness of the plate. The admissible set in this case reads 
\begin{equation}
	\mathbb{A} := [H^2(\Omega \setminus \Sigma) \cap H^1(\Omega)]^3
\end{equation}
to account for the presence of the crease.

The derivation of (\ref{eq:bend_stretch_energy}) is similar to that of (\ref{eq:bend_energy}), but we do not take the limit as the thickness $\theta$ of the plate goes to zero. For this reason, (\ref{eq:bend_stretch_energy}) is referred to as the ``preasymptotic" energy. Note that the scalings in (\ref{eq:bend_stretch_energy}) agree with the results of \cite{Hierarchy_FJM_2006}; a scaling with $\theta$ corresponds to stretching, while a scaling with $\theta^3$ corresponds to bending. Both scalings have been proven rigorously by $\Gamma$-convergence (see \cite{Membrane_DR_1995} for membrane theory and \cite{BLS2016, Bending1_FJM_2002, Bending2_FJM_2002, Bending3_FJMM_2003, lewicka2011} for bending theory). 

We discretize \eqref{eq:bend_stretch_energy} using a continuous, piecewise polynomial space and assume that the edges of the subdivision exactly capture creases when present. The Hessian $D^2 {\bf y}$ in the bending energy is replaced with a discrete Hessian $H_h({\bf y}_h)$. This discrete Hessian, defined in Section 3.2, extends the jump of $\nabla {\bf y}_h$ over the (non-crease) edges of each mesh element to all of $\Omega$ and is essential to proving $\Gamma$-convergence. 

Using the above discretization (see Section~\ref{sec:FEM} for details and notation), we get the following discrete energy:
\begin{equation} \label{eq:bend_stretch_energy_discrete}
	E_h({\bf y}_h) = E_h^S({\bf y}_h) + \theta^2 E_h^B({\bf y}_h),
\end{equation}
where
\begin{equation}
	E_h^S({\bf y}_h) := \frac{1}{8}\int_{\Omega} \Big( 2\mu |{\bf g}^{-1/2}(\nabla {\bf y}_h^T \nabla {\bf y}_h-{\bf g}){\bf g}^{-1/2}|^2 + \lambda \text{tr}({\bf g}^{-1/2}(\nabla {\bf y}_h^T \nabla {\bf y}_h - {\bf g}){\bf g}^{-1/2})^2 \Big) 
\end{equation}
and
\begin{equation}
	\begin{split}
		E_h^B({\bf y}_h) := &\frac{1}{24} \int_{\Omega} \Big( 2\mu |{\bf g}^{-1/2} H_h({\bf y}_h){\bf g}^{-1/2}|^2 + \frac{2\mu \lambda}{2\mu + \lambda} \text{tr}({\bf g}^{-1/2} H_h({\bf y}_h) {\bf g}^{-1/2})^2 \Big)\\
		&+ \frac{\gamma}{2} ||\h^{-1/2}[\nabla {\bf y}_h]||_{L^2(\Gamma_h^0{\setminus \Sigma})}^2.
	\end{split}
\end{equation}
Here $\gamma>0$ is a stabilization parameter, $\Gamma_h^0$ is the skeleton of the underlying shape-regular mesh $\mathcal{T}_h$, and $\h$ represents the mesh function defined in equation \eqref{eqn:mesh_function}.

We note that a slight abuse of the definition of $\Gamma$-convergence is made. In this work $\Gamma$-convergence of the discrete energies $E_h$ towards the continuous energy $E$ is a framework consisting of the lim-inf and lim-sup properties provided in Theorems~\ref{thm:liminf} and~\ref{thm:limsup}. Together with the compactness guaranteed by Lemma~\ref{lemma:compactness}, they guarantee (Theorem~\ref{thm:Gamma_conv}) the convergence of discrete minimzers of $\{ E_h \}_{h>0}$ towards a minimizer of $E$.

\subsection{Contributions}
This paper includes several novel contributions. First, it extends the results of \cite{Prestrain_theoretical_BGNY} to the case with a bending and stretching energy. Second, it introduces folding along a crease, and proves in this setting the $\Gamma$-convergence of the discrete energy, along with the weak and strong convergence properties of the discrete Hessian. The case of folding with a stretching energy is of interest because stretching may be required to capture some physical effects when folding is present. For example, consider the flapping device proposed by Bartels, Bonito, and Hornung \cite{bonitoMc} and depicted in Figure \ref{fig:preasym_flapping}. The device is constructed by folding a sheet of paper down the middle and then pushing one end of the paper upward. When the two corners on the left are pushed together (see the arrows in the figure), the opposite end flaps. The flapping device can be simulated using either model described above with metric ${\bf g}={\bf I}_2$, since the paper is assumed to have no prestrain. However, we do not observe the flapping behavior using the model in \cite{Prestrain_BGNY_2022}, which does not allow for stretching, but only with the preasymptotic model (See Figure \ref{fig:preasym_flapping} and Section~\ref{s:flapping} for more details). 

\begin{figure}[htbp!]
	\begin{center}
		\includegraphics[width=2.50cm]{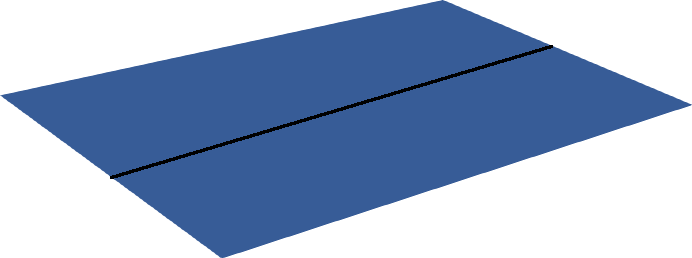}
		\hspace{0.7cm}
		\includegraphics[width=2.75cm]{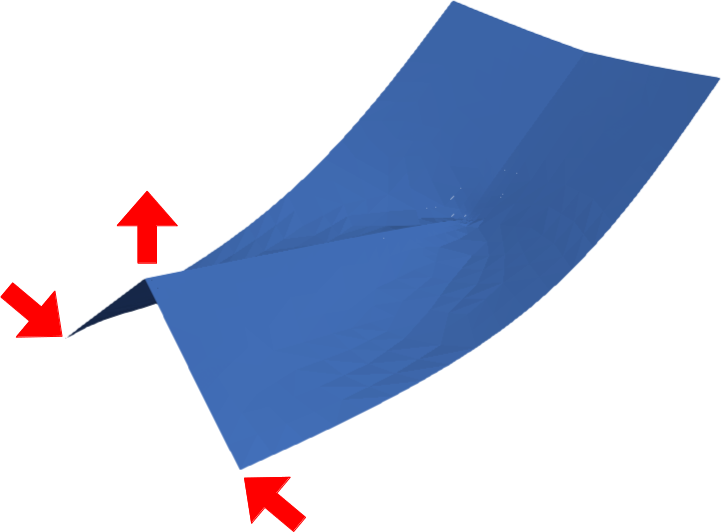}
		\hspace{0.7cm}
		\includegraphics[width=2.50cm]{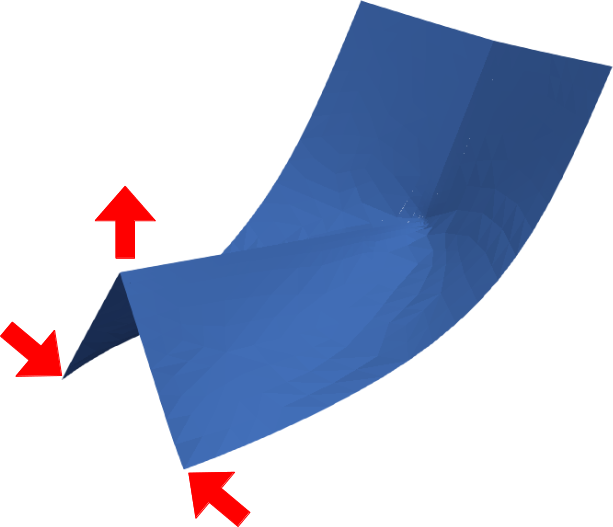} 
		\hspace{0.7cm}
		\includegraphics[width=2.50cm]{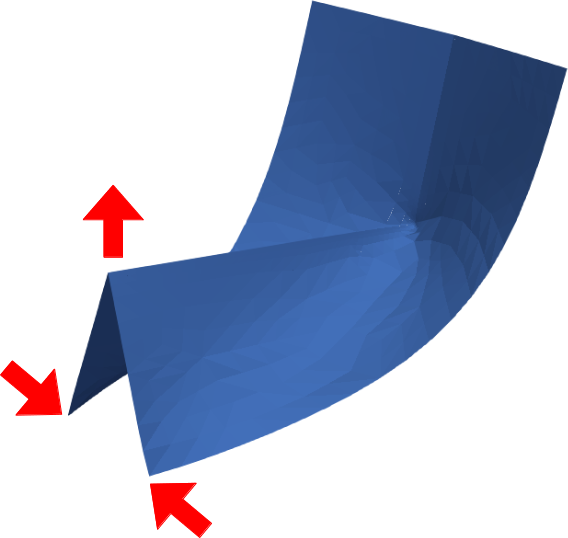} \\
		\vspace{1cm}
		\includegraphics[width=2.50cm]{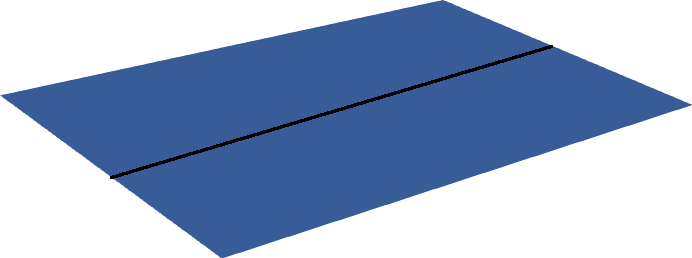}
		\hspace{0.5cm}
		\includegraphics[width=2.65cm]{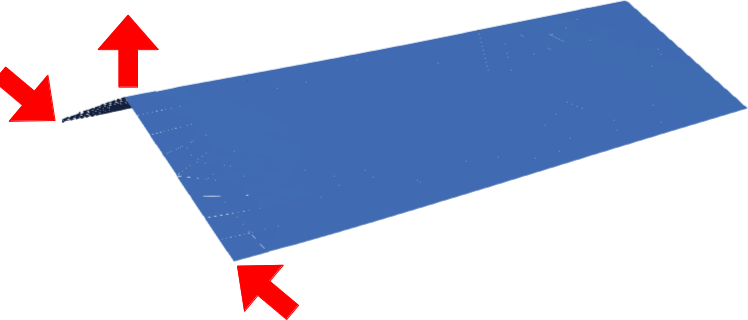}
		\hspace{0.5cm}
		\includegraphics[width=2.65cm]{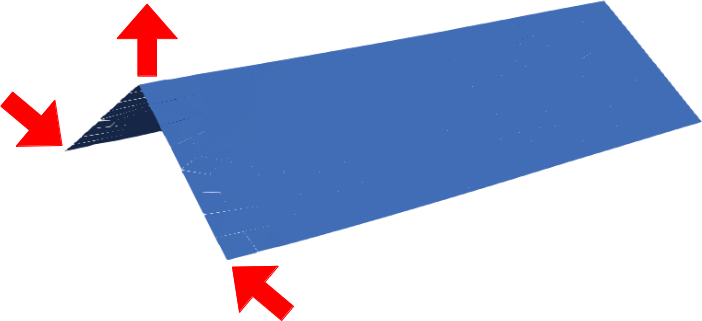}
		\hspace{0.5cm}
		\includegraphics[width=2.65cm]{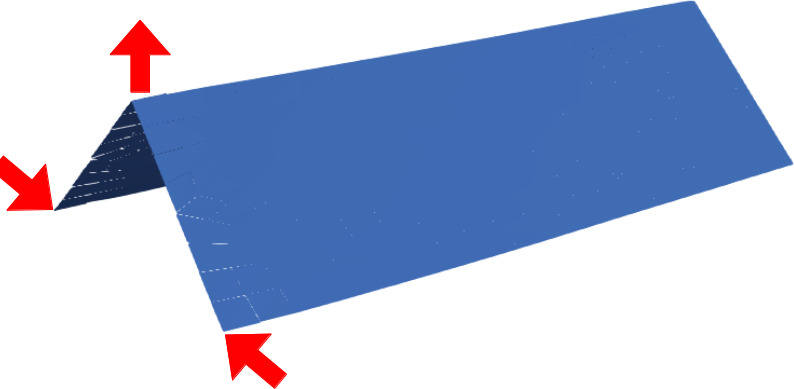}
	\end{center}
	\caption{Evolution of the flapping device as the ends are squeezed. The top row shows the (physically correct) evolution using the preasymptotic model. The bottom row is the evolution using the algorithm proposed in \cite{Prestrain_BGNY_2022} for the minimization of bending energies with isometry constraint.}
	\label{fig:preasym_flapping}
\end{figure}

Introducing the stretching energy also offers advantages because we may use a continuous finite element space to save degrees of freedom, and the minimization algorithm does not require preprocessing to find an initial iterate that satisfies a prescribed metric constraint. Finally, we improve the gradient flow from \cite{BGM_2022} for minimizing (\ref{eq:bend_stretch_energy_discrete}) by treating the stretching term explicitly, which significantly decreases the run-time of simulations. 

\subsection{Outline}
The rest of this work is organized as follows. The derivation of the preasymptotic energy is discussed in Section 2. In Section 3, we introduce a continuous  finite element space. We discuss the discrete Hessian in Section 3.2, along with weak and strong convergence results that are essential for establishing $\Gamma$-convergence of the discrete energy. This discrete energy and its properties are introduced in Section 3.3. Section 4 contains the main result of this paper, namely the $\Gamma$-convergence of the discrete preasymptotic energy to the continuous preasymptotic energy. In Section 5, we present an energy-decreasing gradient flow for minimizing the discrete energy. Finally, Section 6 is devoted to numerical experiments demonstrating the convergence of the discrete minimizers. 

\subsection{Notation}
Bold uppercase letters and bold lowercase letters are used for matrix-valued and vector-valued functions, respectively, and subindices will denote their components. The components of the gradient and the Hessian of a vector-valued function ${\bf v}:\mathbb R^m \rightarrow \mathbb R^n$ are $(\nabla {\bf v})_{ij} := \partial_j v_i$ and $(D^2{\bf v})_{ijk} := \partial_{jk} v_i$, $i=1,..,n$, $j,k=1,...,m$. The Euclidean norm of a vector is denoted $|\cdot|$, while for matrices ${\bf A}, {\bf B}\in\mathbb{R}^{n\times m}$, we write ${\bf A}:{\bf B}:={\rm tr}({\bf B}^T{\bf A})=\sum_{i=1}^{n}\sum_{j=1}^{m}A_{ij}B_{ij}$ and $|{\bf A}|:=\sqrt{{\bf A}:{\bf A}}$. 
Finally, to have a compact notation, for higher-order tensors ${\bf A}=({\bf A}_k)_{k=1}^n\in\mathbb{R}^{n\times m\times m}$ with ${\bf A}_k\in\mathbb{R}^{m\times m}$, $1\le k\le n$, we set
\begin{equation*}
	{\rm tr}({\bf A}) = \big( {\rm tr}({\bf A}_k) \big)_{k=1}^n \quad \mbox{and} \quad |{\bf A}| = \left(\sum_{k=1}^n |{\bf A}_k|^2 \right)^{\frac{1}{2}},
\end{equation*}
and we use the convention
\begin{equation}\label{eqn:conv_tensor}
	{\bf B}{\bf A}{\bf B} := ({\bf B}{\bf A}_k{\bf B})_{k=1}^3  \in \mathbb R^{3 \times 2 \times 2}
\end{equation}
for ${\bf B} \in \mathbb R^{2\times 2}$. Finally, the notation $a\lesssim b$ stands for $a\le Cb$ for a generic constant $C>0$ independent of the discretization parameters and the thickness $\theta>0$.

\section{The Preasymptotic Model for Prestrained Plates}

\subsection{Problem Setting and Preliminaries} \label{section:prelim}
We consider a solid, undeformed plate occupying a region $\Omega_{\theta} = \Omega \times (-\theta/2, \theta/2)$, where $\theta > 0$ is the thickness of the plate and $\Omega \subset \mathbb{R}^2$ is its midplane. Additionally, we introduce a curve $\Sigma \subset \Omega$ representing a crease, where $\Sigma=\emptyset$ if no folding is allowed. For simplicity, we put ourselves in the setting of \cite{BBH2021} and assume that $\Sigma$ is a Lipschitz curve intersecting $\partial \Omega$ transversely and dividing $\Omega$ into two (open) subdomains, denoted $\Omega_1$ and $\Omega_2$. Moreover, we assume that the material has been weakened around the crease so that folding requires no energy (see \cite{BBH2021} for details).
We denote deformations of the plate by ${\bf u}: \Omega_{\theta} \to \mathbb{R}^3$ and deformations of the midplane by ${\bf y}: \Omega \to \mathbb{R}^3$. In this work, thin structures are considered; that is, we assume that $\theta$ is small. As we shall see, this allows us to represent the deformed plate ${\bf u}(\Omega_\theta)$ by the deformation of its two dimensional midplane ${\bf y}(\Omega)$. 

\subsection{Energy Derivation}
The derivation of the preasymptotic model follows closely the derivation of the prestrain model in \cite{Prestrain_BGNY_2022}, with the key difference that we do not take the limit as the thickness of the plate goes to zero. Rather, we assume a small constant thickness $\theta$.
As anticipated in the introduction, the plate energies are composed of competing bending and stretching energies with different scalings in the thickness $\theta$, see \eqref{eq:bend_stretch_energy}. In particular, these plates can stretch and shear.

The reference, or prestrained, metric characterizes the intrinsic stress-free configuration of the plate and is given by 
${\bf G}: \Omega_{\theta} \to  \mathbb{R}^{3 \times 3}$. We assume that ${\bf G}$ is a Riemannian metric and is thus symmetric and uniformly positive definite. It is assumed to be uniform throughout the thickness of the plate. We assume that stretching only occurs along the midplane of the plate, and thus ${\bf G}$ has the following form:
\begin{equation}
	{\bf G}({\bf x}',x_3) := \begin{bmatrix} {\bf g}({\bf x}') &0 \\ 0 & 1  \end{bmatrix}, \qquad {\bf x}':=(x_1,x_2) \in \Omega, \,\, x_3\in\left(-\frac{\theta}{2},\frac{\theta}{2}\right), 
\end{equation}
where ${\bf g}:\Omega \to \mathbb{R}^{2 \times 2}$ is symmetric, uniformly positive definite, and assumed to be in $[W^{1, \infty}(\Omega)]^{2 \times 2}$. 

In the case where ${\bf G}={\bf I}_3$, the stress-free configuration is flat. However, in some materials, such as leaves and flowers, different rates of growth of the various parts of the material result in internal stresses so that ${\bf G} \neq {\bf I}_3$ \cite{Wrinkling_BK_2014, Review_LM_2021}. In the case of hydrogel discs \cite{KVS_2011, Review_LM_2021}, a non-identity metric is introduced by injecting a heat-sensitive gel. We also note that for ${\bf G} \neq {\bf I}_3$, a stress-free configuration of the material may not exist. In this case, the plate may assume a more complex shape with positive energy.

The thin materials are assumed to be endowed with the St.~Venant--Kirchhoff (isotropic) energy in the internal stress-free configuration, i.e. 
\begin{equation}
	E_{\text{3D}}({\bf u}) := \int_{\Omega \times (-\frac{\theta}{2},\frac{\theta}{2})} W(\nabla {\bf u} {\bf G}^{-1/2}), 
\end{equation}
where
\begin{equation} \label{eqn:W_density}
	W({\bf F}) := \mu | \epsilon({\bf F}) |^2 + \frac{\lambda}{2} \text{tr}(\epsilon({\bf F}))^2
\end{equation}
with $\epsilon({\bf F}) := \frac{1}{2}(F^T F-{\bf I}_3)$ for ${\bf F} \in \mathbb{R}^{3 \times 3}$, and where $\mu>0$ and $\lambda\ge 0$ are the Lam\'e coefficients. The proposed dimensionally reduced model for the midplane hinges on the following assumption relating the deformation ${\bf u}$ of the plate and the deformation ${\bf y}$ of its midplane:
\begin{equation}\label{eq:kirchhoff}
	{\bf u} ({\bf x}',x_3) = {\bf y}({\bf x}') + x_3 {\boldsymbol \nu}({\bf x}') + \frac{1}{2} x_3^2\beta({\bf x}') {\boldsymbol \nu}({\bf x}'),
\end{equation}
where 
\begin{equation} \label{eqn:beta_nu}
	\beta := \frac{\lambda}{2\mu+\lambda} \text{tr}({\bf g}^{-1/2} \II({\bf y}){\bf g}^{-1/2}), \qquad {\boldsymbol \nu}:= \frac{\partial_1 {\bf y} \times \partial_2 {\bf y}}{|\partial_1 {\bf y} \times \partial_2 {\bf y}|},
\end{equation}
and $\II({\bf y})$ is the second fundamental form of the surface ${\bf y}(\Omega)$. The first and second fundamental forms of the surface ${\bf y}(\Omega)$, which describe its geometry, are given respectively by 

\begin{equation} \label{eqn:first_second}
	\I({\bf y}) := \nabla {\bf y}^T \nabla {\bf y} \qquad \text{and} \qquad \II({\bf y}) := -\nabla {\boldsymbol \nu}^T \nabla {\bf y}.
\end{equation}
We consider deformations such that $E_{\text{3D}}({\bf u}) \leq \Lambda \theta^3$ to reflect the bending regime of interest.

The modified Kirchoff-Love assumption \eqref{eq:kirchhoff}, inspired by \cite{Hierarchy_FJM_2006, Bending1_FJM_2002}, restricts the fibers orthogonal to the midplane to remain so during deformations, but allows these fibers to be inhomogeneously stretched \cite{Prestrain_BGNY_2022}. Such a structural assumption on the deformations is consistent with the $\Gamma-$limit of the rescaled energy  $\theta^{-3} E_{\text{3D}}({\bf u})$ as $\theta \to 0^+$, namely this is the ansantz yielding the correct limiting energy in the asymptotic analysis, see for instance Section 7 in \cite{Bending2_FJM_2002} (case ${\bf G}={\bf I}_3$) and Section 2.2 in \cite{Prestrain_BGNY_2022} for details.

Computing the energy density and neglecting $\mathcal{O}(\theta^3)$ terms, the 3D energy per unit volume $\theta^{-1} E_{\text{3D}}({\bf u}) =: E({\bf y})$ reads \cite{Prestrain_BGNY_2022,BGM_2022}:
\begin{equation} \label{eq:energy_final}
	E({\bf y}) = E^S({\bf y}) + \theta^2 E^B({\bf y}),
\end{equation}
where the stretching energy $E^S$ is given by
\begin{equation} \label{eqn:EstrII}
	E^S({\bf y}) := \frac{1}{8}\int_{\Omega} \Big( 2\mu |{\bf g}^{-1/2}( \I ({\bf y})-{\bf g}){\bf g}^{-1/2}|^2 + \lambda \text{tr}({\bf g}^{-1/2}(\I({\bf y}) - {\bf g}){\bf g}^{-1/2})^2 \Big) 
\end{equation}
and the bending energy $E^B$ reads
\begin{equation} \label{eqn:EbendII}
	E^B({\bf y}) := \frac{1}{24} \int_{\Omega} \Big( 2\mu |{\bf g}^{-1/2}\II ({\bf y}){\bf g}^{-1/2}|^2 + \frac{2\mu \lambda}{2\mu + \lambda} \text{tr}({\bf g}^{-1/2}\II({\bf y}){\bf g}^{-1/2})^2 \Big).
\end{equation}

Compared to \cite{Prestrain_BGNY_2022}, the midplane deformations ${\bf y}$ are not restricted to strictly satisfy the constraint $\I({\bf y}) = {\bf g}$. Rather, deviations from this relation are penalized by the stretching energy.

To incorporate folding, we change the domain $\Omega$ in \eqref{eqn:EstrII} and \eqref{eqn:EbendII} to the domain $\Omega\setminus \Sigma$, where $\Sigma$ is the Lipschitz curve described in Section \ref{section:prelim}. Note that this modified energy is currently without justification from 3D hyperelasticity; however, a rigorous derivation in the case where ${\bf G}={\bf I}_3$ and $\theta \to 0^+$ is done in \cite{BBH2021}.

In addition, proceeding as in \cite{Wrinkling_BK_2014} (see also \cite{CM2008}), we replace for convenience the second fundamental form by the Hessian $D^2{\bf y}$ in the bending energy. Such an assumption is not geometrically justified, but it is not expected to significantly influence the equilibrium deformations. We present in Section 6.1 below an experiment numerically justifying this step. We also note that using the Hessian is justified by Proposition 1 in \cite{BNY2023} when the deformation satisfies the metric constraint.
Our final bending and stretching energies thus become
\begin{equation} \label{eqn:Ebend}
	E^B({\bf y}) := \frac{1}{24} \int_{\Omega \setminus \Sigma} \Big( 2\mu |{\bf g}^{-1/2}D^2 {\bf y}{\bf g}^{-1/2}|^2 + \frac{2\mu \lambda}{2\mu + \lambda} \text{tr}({\bf g}^{-1/2}D^2 {\bf y}{\bf g}^{-1/2})^2 \Big).
\end{equation}
and
\begin{equation} \label{eqn:Estr}
	E^S({\bf y}) := \frac{1}{8}\int_{\Omega} \Big( 2\mu |{\bf g}^{-1/2}( \I ({\bf y})-{\bf g}){\bf g}^{-1/2}|^2 + \lambda \text{tr}({\bf g}^{-1/2}(\I({\bf y}) - {\bf g}){\bf g}^{-1/2})^2 \Big),
\end{equation}
for ${\bf y} \in [H^2(\Omega \setminus \Sigma) \cap H^1(\Omega)]^3$.

We note that in the absence of boundary conditions, the energies \eqref{eqn:Ebend} and \eqref{eqn:Estr} are frame indifferent up to rigid motions. In the analysis that follows, we will consider only the free boundary case. However, boundary conditions are required to model various phenomena (e.g., the flapping device of Section \ref{s:flapping}). The results of Sections \ref{sec:FEM} and \ref{sec:gamma_conv} extend to the case where Dirichlet conditions are imposed on some portion $\Gamma_D$ of $\partial \Omega$; we refer the reader to Appendix C of \cite{Prestrain_theoretical_BGNY} for the necessary modifications of the proofs.

We also mention that standard results in calculus of variations guarantee the existence of global minimizers to the energy \eqref{eq:energy_final} where $E^B$ is given by \eqref{eqn:Ebend} and $E^S$ by \eqref{eqn:Estr}. Rather than embarking in this discussion, we point out that the $\Gamma$-convergence theory developed below (Theorem~\ref{thm:liminf}) guarantees the existence of such minimizers.

\section{Finite Element Discretization} \label{sec:FEM}

Let $\{ \mathcal{T}_h \}_{h >0}$ be a shape-regular and quasi-uniform sequence of conforming subdivisions of $\Omega$, where every element $T$ is given by the image of a reference triangle or quadrilateral $\hat{T}$ under an isoparametric diffeomorphism $\psi_T: \hat{T} \to T$ of degree $k$. We assume that the crease $\Sigma$ is resolved by the mesh, and for $i=1,2$, we denote by $\mathcal{T}_h^i$ the collection of the elements that belong to $\Omega_i$, the two parts of $\Omega$ separated by $\Sigma$. The set of interior edges is denoted $\mathcal{E}_h^0$ and we write $\mathcal{E}_h^{\Sigma}$  for the set of edges along $\Sigma$. We also let $\Gamma_h^0 := \{{\bf x} \in e \ | \ e \in \mathcal{E}_h^0 \}$ be the interior skeleton of $\mathcal{T}_h$. Finally, we introduce the mesh function $\h$ defined on $\mathcal{T}_h\cup \mathcal{E}_h^0$ by
\begin{equation} \label{eqn:mesh_function}
	\restriction{\h}{T} := h_{T}, \quad \forall\, T\in\mathcal{T}_h, \qquad \text{and} \qquad \restriction{\h}{e} := h_e, \quad \forall\, e\in\mathcal{E}_h^0,
\end{equation}
where $h_T$ and $h_e$ denote the diameter of $T\in \mathcal{T}_h$ and $e\in \mathcal{E}_h^0$, respectively. 

Our finite element space for each component of the deformation is 
\begin{equation*}
	S_h^k := \{v_h \in C^0(\overline{\Omega}) \ | \ {v_h|_T = \hat{v} \circ \psi_T^{-1}, \hat{v} \in \mathbb{P}_k(\hat{T}) \ (\text{resp.} \, \mathbb{Q}_k(\hat{T}))} \  \forall T \in \mathcal{T}_h \},    
\end{equation*}
where $\mathbb{P}_k{(\hat{T})}$ is the space of polynomials of total degree $k$ and $\mathbb{Q}_k(\hat{T})$ is the space of polynomials of degree $k$ in each direction on the reference element (note that to simplify the presentation, we use an isoparametric method, i.e., the same polynomial degree is used for $\psi_T$ and $\hat v$). An approximation of ${\bf y}$ is a deformation ${\bf y}_h\in [S_h^k]^3$ where, from now on, $k\ge 2$. 

Since $\mathcal T_h$ is shape-regular, we have $||D{ \psi_T}||_{L^{\infty}(\hat{T})} \lesssim h_T$ and $||D{\psi_T^{-1}}||_{L^{\infty}(T)} \lesssim h_T^{-1}$. For higher order derivatives, we recall the estimate  \cite{L1986} 
\begin{equation} \label{eq:iso_estI}
	||D^m \psi_T||_{L^{\infty}(\hat{T})} \lesssim h_T^m, \qquad 2 \leq m \leq k+1.
\end{equation}
Since $||D^2 \psi_T||_{L^{\infty}(\hat{T})} \lesssim ||D \psi_T||_{L^{\infty}(\hat{T})}$, the following estimates on the inverse mapping hold \cite{EG2021}:
\begin{equation} \label{eq:iso_estII}
	||D^m \psi_T^{-1}||_{L^{\infty}(T)} \lesssim h_T^{-m}, \qquad 2 \leq m \leq k+1.
\end{equation}

Although a function $v_h \in S_h^k$ is continuous over the domain $\Omega$, the gradient $\nabla v_h$ is discontinuous on $\Gamma_h^0$. Thus, $\nabla v_h$ has, in general, nonzero jump over each edge of $\mathcal{E}_h^0$. We make these notions more precise now. For any function $w_h$ defined on $\Omega$, for $e\in{\mathcal E}_h^0$ and ${\bf x} \in e$ we set $w_h^{\pm}({\bf x }):=\lim_{s\to 0^+}w_h({\bf x} \pm s{\bf n}_e)$, where ${\bf n}_e$ is a fixed unit normal to $e$.  The jump of $w_h$ over $e$ is then given by
\begin{equation}\label{eq:jump}
	\restriction{[w_h]}{e} := w_h^-- w_h^+ 
\end{equation}
and the average is defined as 
\begin{equation} \label{eq:avg}
	\restriction{\{w_h\}}{e} := \frac{1}{2}(w_h^+ + w_h^-).
\end{equation}
For a vector or matrix valued function, the jump and average are defined componentwise. In paricular, for ${\bf v}_h \in [S_h^k]^3$, the jump and average of $\nabla {\bf v}_h$ are defined componentwise with $w_h = \partial_j v_{h,i}$, $i=1,2,3$, $j=1,2$.

Using the jump operator, we define the bilinear form $\langle \cdot, \cdot \rangle_{H_h^2(\Omega)}$ on $S_h^k$ by
\begin{equation}\label{eq:bilinear_form}
	\langle v_h, w_h \rangle_{H_h^2(\Omega)} := (D_h^2v_h, D_h^2 w_h )_{L^2(\Omega)} + (\h^{-1} [\nabla v_h], [\nabla w_h])_{L^2(\Gamma_h^0 \setminus \Sigma)}
\end{equation}
for $v_h, w_h \in S_h^k$, where $D_h^2$ is the broken (aka piecewise) Hessian, and the seminorm $| \cdot |_{H_h^2(\Omega)}$ is given by
\begin{equation}
	|v_h|_{H_h^2(\Omega)}^{2} := \langle v_h, v_h \rangle_{H_h^2(\Omega)}
\end{equation}
for $v_h \in S_h^k$. 
From this we define the scalar product
\begin{equation}
	(v_h, w_h)_{H_h^2(\Omega)} := \langle v_h, w_h \rangle_{H_h^2(\Omega)} + (v_h, w_h)_{L^2(\Omega)}
\end{equation}
and the norm 
\begin{equation} \label{def:discrete_H2_norm}
	||v_h||_{H_h^2(\Omega)}^2 := (v_h, v_h)_{H_h^2(\Omega)}.
\end{equation}
For ${\bf v}_h$, ${\bf w}_h \in [S_h^k]^3$ we define $\langle {\bf v}_h, {\bf w}_h \rangle_{H_h^2(\Omega)} = \sum_{i=1}^3 \langle v_{h,i}, w_{h,i}\rangle$, and similarly for the seminorm, scalar product, and norm.

\subsection{Compactness}
We next prove a compactness result that will be used in the proof of $\Gamma$-convergence. In order to establish this result, we need the following smoothing interpolation operator $\Pi_h$ inspired by \cite{Prestrain_theoretical_BGNY}, but applied on each domain $\Omega_i$, $i=1,2$, and adapted to the isoparametric case: 
\begin{equation} \label{def:Pi_h}
	\text{for } v\in \mathbb{E}(\mathcal{T}_h):= \prod_{T \in \mathcal{T}_h} H^1(T), \quad \restriction{(\Pi_h v)}{\Omega_i}:=\Pi_h^i(\restriction{v}{\Omega_i})\in S_h^k(\Omega_i), \quad i=1,2,    
\end{equation}
where $S_h^k(\Omega_i)$ is the restriction of $S_h^k$ to $\Omega_i$ and $\Pi_h^i$ is defined in two steps as follows. First, we construct $P_h^iv\in L^2(\Omega_i)$ given elementwise by $\restriction{(P_h^iv)}{T} = \hat P(v\circ\psi_T)\circ\psi_T^{-1}$ with $\hat P$ the $L^2(\hat{T})$ projection onto $\mathbb{P}_k(\hat{T})$ ($\mathbb{Q}_k(\hat T)$ for quadrilaterals). Then we apply a Cl\'ement (or Scott-Zhang) interpolant $I_h^i:L^2(\Omega_i)\rightarrow S_h^k(\Omega_i)$ to $P_h^iv$ defined via local $L^2$ projections on references patches.

The operator $\Pi_h$ satisfies for $v \in \mathbb{E}(\mathcal{T}_h)$
\begin{equation}\label{eq:interp_prop_1}
	||\Pi_h v ||_{L^2(\Omega)} \lesssim ||v||_{L^2(\Omega)}
\end{equation}
and
\begin{equation}\label{eq:interp_prop_2}
	||\nabla_h \Pi_h v||_{L^2(\Omega)} + ||\h^{-1}(v-\Pi_h v)||_{L^2(\Omega)} \lesssim ||\nabla_h v||_{L^2(\Omega)} + ||\h^{-1/2}[v]||_{L^2(\Gamma_h^0 \setminus \Sigma)},
\end{equation}
where $\nabla_h$ denotes the broken gradient. We omit the proofs of these relations since they follow from the arguments provided in \cite{Prestrain_theoretical_BGNY,Conv_adaptive_dG_BN_2010,DG_largebending_BNN_2021} applied to each subdomain $\Omega_i$ and adding the resulting two inequalities.

We also need the following discrete Poincar{\' e}--Friedrichs inequalities. Such inequalities are proven in \cite{Prestrain_theoretical_BGNY}, where the approximation space $S^k_h$ is fully discontinuous, but the argument applies directly to the current context. 

\begin{lemma}[Discrete Poincar{\'e}--Friedrichs inequality] \label{lemma:disc_poincare}
	For any $v \in \mathbb{E}(\mathcal{T}_h)$, 
	\begin{equation} \label{eq:disc_poincare_1}
		||v - \strokedint_{\Omega} v ||_{L^2(\Omega)} \lesssim ||\nabla_h v||_{L^2(\Omega)}+||\h^{-1/2}[v]||_{L^2(\Gamma_h^0)},
	\end{equation}
	where $\strokedint_{\Omega} v$ denotes the average of $v$ over $\Omega$.
	Also, for $v_h \in S_h^k$ we have
	\begin{equation} \label{eq:disc_poincare_2}
		||\nabla v_h||_{L^2(\Omega)}+||\nabla_h\Pi_h v_h||_{L^2(\Omega)} \lesssim \|v_h\|_{H_h^2(\Omega)}.
	\end{equation}
\end{lemma}

Now we can verify the following compactness result.

\begin{lemma}[Compactness] \label{lemma:compactness}
	Assume that $\{ v_h \}_{h>0} \subset S_h^k$ is a sequence such that 
	\begin{equation}\label{eq:bound}
		||\nabla v_h||_{L^2(\Omega)} + |v_h|_{H^2_h(\Omega)} \lesssim 1.
	\end{equation}
	Then there exists $\bar{v} \in H^2(\Omega \setminus \Sigma)\cap H^1(\Omega)$ with mean value zero such that (up to a subsequence not indicated here) 
	$$
	\bar v_h := v_h - \strokedint_{\Omega} v_h \to \bar{v} \quad \textrm{in}\quad L^2(\Omega) \quad \textrm{and}\quad \nabla v_h \to \nabla \bar{v} \quad \textrm{in}\quad [L^2(\Omega)]^{2} \quad \textrm{as }h \to 0^+.
	$$
\end{lemma}
\begin{proof}
	The proof follows those of Lemma 2.2 in \cite{Prestrain_theoretical_BGNY} and Proposition 5.1 in \cite{DG_largebending_BNN_2021}, which were written in the context of fully discontinuous approximation spaces.
	Let  $c_h := \strokedint_{\Omega} v_h$. By the standard Poincar{\' e}-Friedrichs inequality and the boundedness assumption \eqref{eq:bound}, we have 
	$$
	||v_h- c_h ||_{L^2(\Omega)} + ||\nabla v_h||_{L^2(\Omega)} \lesssim ||\nabla v_h||_{L^2(\Omega)} \lesssim 1.$$
	Thus, $v_h - c_h$ is uniformly bounded in $H^1(\Omega)$ and therefore converges strongly (up to a subsequence) in $L^2(\Omega)$ to some $\bar{v} \in H^1(\Omega)$.
	
	To show that $\nabla v_h \to \nabla \bar{v}$, we make use of the smoothing interpolation operator defined above. Let ${\bf z}_h := \Pi_h (\nabla v_h)$ with $\Pi_h$ applied componentwise. The stability estimates (\ref{eq:interp_prop_1}) and (\ref{eq:interp_prop_2}) together with the boundedness assumption \eqref{eq:bound} again yield
	\begin{equation}
		||{\bf z}_h||_{L^2(\Omega)} + ||\nabla_h {\bf z}_h||_{L^2(\Omega)} \lesssim ||\nabla v_h||_{L^2(\Omega)} +  ||D_h^2 v_h||_{L^2(\Omega)} + ||\h^{-1/2}[\nabla v_h]||_{L^2(\Gamma_h^0 \setminus \Sigma)} \lesssim 1.
	\end{equation}
	This implies that ${\bf z}_h$ is uniformly bounded in $[H^1(\Omega  \setminus \Sigma)]^{2}$ and thus converges strongly (up to a subsequence) in $[L^2(\Omega)]^{2}$ to some ${\bf z} \in [H^1(\Omega \setminus \Sigma)]^{2}$.
	
	We now establish the convergence of $\nabla v_h$ to ${\bf z}$ in $[L^2(\Omega)]^{2}$. Recalling the notation $h = \max_{{\bf x} \in \Omega} \h({\bf x})$ and using the stability estimate (\ref{eq:interp_prop_2}) along with the uniform boundedness assumption \eqref{eq:bound}, we get
	\begin{align*}
		||\nabla v_h - {\bf z}_h||_{L^2(\Omega)} &\le h ||\h^{-1}(\nabla v_h - {\bf z}_h)||_{L^2(\Omega)} \\
		&\lesssim h \Big(||D_h^2 v_h||_{L^2(\Omega)} + ||\h^{-1/2}[\nabla v_h]||_{L^2(\Gamma_h^0 \setminus \Sigma)} \Big) \\
		&\lesssim h.
	\end{align*}
	Therefore, we have
	\begin{equation*}
		||\nabla v_h - {\bf z}||_{L^2(\Omega)} \leq ||\nabla v_h - {\bf z}_h||_{L^2(\Omega)} + ||{\bf z}_h -{\bf z}||_{L^2(\Omega)} \to 0
	\end{equation*}
	as $h \to 0^+$, and thus $\nabla v_h \to {\bf z}$ in $[L^2(\Omega)]^2$ as $h \to 0^+$. This and the fact that $v_h-c_h \to \bar{v}$ in $L^2(\Omega)$ imply that ${\bf z} = \nabla \bar{v}$, completing the proof.
\end{proof}

We end this section with a discrete Sobolev inequality, which is instrumental for handling the nonconvex stretching energy.

\begin{lemma} [Discrete Sobolev inequality] \label{lemma:disc_sobolev}
	For ${\bf v}_h \in [S_h^k]^3$, the following inequality holds
	\begin{equation} \label{eq:disc_sobolev}
		||\nabla {\bf v}_h ||_{L^4(\Omega)} \lesssim ||{\bf v}_h||_{H_h^2(\Omega)}.
	\end{equation}
\end{lemma}
\begin{proof}
	We establish the following for all $v \in \mathbb{E}(\mathcal{T}_h)=\prod_{T \in \mathcal T_h} H^1(T)$: 
	\begin{equation} \label{eq:intermediate_disc_sobolev}
		||v||_{L^4(\Omega)} \lesssim ||\nabla_h v ||_{L^2(\Omega)} + ||v||_{L^2(\Omega)} +||\h^{-1/2}[v]||_{L^2(\Gamma_h^0  \setminus \Sigma)}.
	\end{equation}
	The desired estimate \eqref{eq:disc_sobolev} follows from \eqref{eq:intermediate_disc_sobolev} by applying the latter to each component of $\nabla{\bf v}_h \in [\mathbb{E}(\mathcal{T}_h)]^{3\times 2}$ and invoking the Poincar{\'e}-Friedrichs inequality \eqref{eq:disc_poincare_2} on each subdomain $\Omega_i$, $i=1,2$.
	
	Thus, it remains to prove \eqref{eq:intermediate_disc_sobolev}.
	We write
	$$
	||v||_{L^4(\Omega)} \leq ||v-\Pi_h v||_{L^4(\Omega)} + ||\Pi_h v||_{L^4(\Omega)},
	$$
	where $\Pi_h$ is the operator defined in \eqref{def:Pi_h}. The inverse inequality $||v-\Pi_h v||_{H^1(\Omega)} \lesssim ||\h^{-1}(v-\Pi_h v)||_{L^2(\Omega)}$, together with the embedding $H^1(\Omega) \subset L^4(\Omega)$ yields
	$$
	||v||_{L^4(\Omega)}   \lesssim ||\h^{-1}(v-\Pi_h v)||_{L^2(\Omega)} + ||\Pi_h v||_{H^1(\Omega)}.
	$$
	Whence, the intermediate estimate \eqref{eq:intermediate_disc_sobolev} follows from the stability estimates (\ref{eq:interp_prop_1}) and (\ref{eq:interp_prop_2}).
\end{proof}

\subsection{Discrete Hessian}
Now we are ready to define the discrete, or reconstructed, Hessian which was originally developed for the bending of plates in \cite{Prestrain_BGNY_2022} and \cite{DG_largebending_BNN_2021} using ideas from \cite{DPE2010} and \cite{DPE2011}, see also \cite{LP_2011,P_2014}. The key property of the discrete Hessian is that it mimics $\widetilde{D}^2$, the restriction of the Hessian $D^2$ to $\Omega \setminus \Sigma$.

The definition of the discrete Hessian is based on local lifting operators that extend the jumps on $\Gamma_h^0\setminus\Sigma$ to all of $\Omega$. Let $\ell$ be a non-negative integer and recall that the average of a function is defined in \eqref{eq:avg}. On each interior and non-crease edge $e \in \mathcal{E}_h^0 \setminus \mathcal{E}_h^{\Sigma}$, we define the local lifting $r_e : [L^2(e)]^2 \to [V_h^{\ell}]^{2 \times 2}$ by the relation

\begin{equation}
	\int_{\omega_e} r_e({\bf \phi}) : \boldsymbol{\tau}_h = \int_e \{\tau_h\}{\bf n}_e \cdot {\bf \phi} \ \  \forall \boldsymbol{\tau}_h \in [V_h^{\ell}]^{2 \times 2},
\end{equation}
where $\omega_e$ is union of the two elements sharing $e$ and $V_h^{\ell}$ is the broken finite element space
\begin{equation*}
	V_h^{\ell} := \{ v_h \in L^2(\Omega) \ |  \ v_h|_T = \hat{v} \circ \psi_T^{-1}, \hat{v} \in \mathbb{P}_{\ell}(\hat{T}) \ (\text{resp. } \mathbb{Q}_{\ell}(\hat{T})) \ \forall T \in \mathcal{T}_h  \}.
\end{equation*}
The corresponding global lifting operator $R_h : [L^2(\Gamma_h^0\setminus\Sigma)]^2 \to [V_h^{\ell}]^{2 \times 2}$ is defined  as
\begin{equation}
	R_h := \sum_{e \in \mathcal{E}_h^0  \setminus \mathcal{E}_h^{\Sigma}} r_e.
\end{equation}

The global lifting operator $R_h$ is stable in the sense that for any $v_h \in S_h^k$ and for any $\ell \geq 0$, 
\begin{equation}\label{eq:lift_stability}
	||R_h([\nabla v_h])||_{L^2(\Omega)} \lesssim ||\h^{-1/2}[\nabla v_h]||_{L^2(\Gamma_h^0  \setminus \Sigma)}.
\end{equation}
The proof of \eqref{eq:lift_stability} follows the proof of Lemma 4.34 in \cite{DPE2011} up to minor modifications and is therefore omitted.

The global lifting operator supplements the piecewise Hessian $D^2_h$ to define the discrete Hessian operator $H_h : S_h^k \to [L^2(\Omega)]^{2 \times 2}$:
\begin{equation}
	H_h(v_h) := D_h^2 v_h - R_h([\nabla v_h]). 
\end{equation}
This definition can also be extended to $[S_h^k]^3$ by component-wise application. 
Note that unlike $\widetilde{D}^2 v_h$, which is a distribution, all the components of $H_h(v_h)$ are in $L^2(\Omega)$. This aspect is the key to deducing convergence properties of $H_h$ towards $\widetilde{D}^2$ as $h\to 0^+$, see Lemmas~\ref{lemma:weak_conv_hessian} and~\ref{lemma:strong_conv_hessian} below. 
In addition, it is worth pointing out that each component of $D^2_h v_h|_T$ belongs to $\mathbb P_{k-2}$ when $T$ is a triangle, but to $\mathbb Q_k$ when $T$ is a quadrilateral. In practice we set $\ell=k$, but the results presented below are valid for all choices $\ell \geq 0$. 

As already anticipated, one of the advantages of the discrete Hessian is that, unlike the broken Hessian $D_h^2 {\bf y}_h$, $H_h({\bf y}_h)$ converges weakly to
$\widetilde{D}^2 {\bf y}$ in $[L^2(\Omega)]^{3 \times 2 \times 2}$ if ${\bf y}_h$ converges to ${\bf y}$ in $[L^2(\Omega)]^3$. In addition, the discrete Hessian $H_h({\bf y}_h)$ of the Lagrange interpolant ${\bf y}_h$ of ${\bf y}\in[H^2(\Omega \setminus\Sigma)\cap H^1(\Omega)]^3$ defined separately on each subdomain strongly converges to $\widetilde{D}^2{\bf y}$. Both properties are essential for establishing $\Gamma$-convergence of the discrete energy \eqref{def:Eh} defined below to \eqref{eq:energy_final}. Our aim in the rest of this subsection is to establish these weak and strong convergence properties of the discrete Hessian.
Note that both Lemma \ref{lemma:weak_conv_hessian} and Lemma \ref{lemma:strong_conv_hessian} below can be applied componentwise for ${\bf v}_h \in [S_h^k]^3$.

\begin{lemma}[Weak convergence of $H_h$] \label{lemma:weak_conv_hessian}
	Let $\{ v_h \}_{h >0} \subset S_h^k$ be such that $|v_h|_{H_h^2(\Omega)} \lesssim 1$ and $v_h \to v$ in $L^2(\Omega)$ as $h \to 0^+$ for some $v \in H^2(\Omega \setminus \Sigma)\cap H^1(\Omega)$. Then for any $\ell \geq 0$, we have
	\begin{equation}
		H_h(v_h) \rightharpoonup \widetilde{D}^2 v \quad \text{in} \quad  [L^2(\Omega)]^{2 \times 2} \qquad \textrm{as}\quad h\to 0^+.
	\end{equation}
\end{lemma}
\begin{proof}
	We prove the convergence on each subdomain $\Omega_i$, $i=1,2$. It is enough to consider functions in $[C_0^{\infty}(\Omega_i)]^{2\times 2}$ since this space is dense in $[L^2(\Omega_i)]^{2\times 2}$. Let $\boldsymbol{\phi} \in [C_0^{\infty}(\Omega_i)]^{2\times 2}$.  Integrating by parts twice gives
	\begin{align*}
		\int_{\Omega_i} &H_h(v_h) : \boldsymbol{\phi} = \int_{\Omega_i} D_h^2 v_h : \phi - \int_{\Omega_i} R_h([\nabla v_h]):\boldsymbol{\phi} \\
		&= \int_{\Omega_i} v_h (\text{div } \text{div} \boldsymbol{\phi})- \int_{\Omega_{i}} R_h([\nabla v_h]):(\boldsymbol{\phi}-\boldsymbol{\phi}_h) + \sum_{e \in \mathcal{E}_h^0  \cap \Omega_i} \int_e [\nabla v_h] \cdot \{\boldsymbol{\phi} - \boldsymbol{\phi}_h \} {\bf n}_e \\
		&=: I_1 + I_2 + I_3,
	\end{align*}
	where $\boldsymbol{\phi}_h=\mathcal{I}_h^{k} \boldsymbol{\phi} \in [S_h^{k}]^{2\times 2}$ denotes the Lagrange interpolant of $\boldsymbol{\phi}$. We now deal with $I_1$, $I_2$, and $I_3$ in turn.
	
	For $I_1$, it suffices to invoke the strong convergence $v_h \to v$ in $L^2(\Omega)$ to deduce that
	$$\int_{\Omega_i} v_h (\text{div } \text{div} \boldsymbol{\phi}) \to \int_{\Omega_i} v (\text{div } \text{div} \boldsymbol{\phi}) = \int_{\Omega_i} D^2 v : \boldsymbol{\phi} \qquad \textrm{as}\quad h \to 0^+.$$
	
	For $I_2$, the boundedness assumption $|v_h|_{H_h^2(\Omega)} \lesssim 1$ and the stability of the global lifting operator \eqref{eq:lift_stability} imply that 
	$$ |I_2|  
	\lesssim ||\h^{-1/2}[\nabla v_h]||_{L^2( \Gamma_h^0 \cap \Omega_i)}||\boldsymbol{\phi}- \boldsymbol{\phi}_h||_{L^2(\Omega_i)}  \\
	\lesssim ||\boldsymbol{\phi} - \boldsymbol{\phi}_h||_{L^2(\Omega_i)} \to 0$$
	as $h \to 0^+$.
	
	For $I_3$, the scaled trace inequality yields, for $e\in\mathcal{E}_h^0 \cap \Omega_i$,
	\begin{equation}\label{eq:scaled_trace}
		||\boldsymbol{\phi} - \boldsymbol{\phi}_h||_{L^2(e)} \lesssim ||\h^{-1/2}(\boldsymbol{\phi} - \boldsymbol{\phi}_h)||_{L^2(\omega_e)} + ||\h^{1/2}\nabla (\boldsymbol{\phi} - \boldsymbol{\phi}_h) ||_{L^2(\omega_e)}.
	\end{equation}
	Whence, the shape-regularity of $\{\mathcal{T}_h\}_{h>0}$ and the boundedness assumption $|v_h|_{H_h^2(\Omega)} \lesssim 1$ imply 
	\begin{equation*}
		|I_3| \lesssim \Big( \sum_{e \in \mathcal{E}_h^0 {\cap \Omega_i}} ||\h^{-1/2}[\nabla v_h]||_{L^2(e)}^2 \Big)^{1/2} \Big(||\boldsymbol{\phi} - \boldsymbol{\phi}_h||_{L^2(\Omega_i)} + ||\h \nabla (\boldsymbol{\phi} - \boldsymbol{\phi}_h)||_{L^2(\Omega_i)} \Big) \to 0
	\end{equation*}
	as $h \to 0^+$.
	
	The combined results for $I_1$, $I_2$, and $I_3$ give the desired weak convergence.
\end{proof}

\begin{lemma}[Strong convergence of $H_h$]\label{lemma:strong_conv_hessian}
	Let $v \in H^2(\Omega \setminus \Sigma) \cap H^1(\Omega)$ and $v_h = \mathcal{I}_h^k v$ be the Lagrange interpolant of $v$ defined separately on each subdomain. Then for any $\ell \geq 0$, we have
	\begin{equation}
		D_h^2 v_h \to \widetilde{D}^2 v \quad \text{and} \quad R_h([\nabla v_h]) \to 0 \quad \text{in}\quad  [L^2(\Omega)]^{2 \times 2} \qquad \textrm{as}\quad h \to 0^+.
	\end{equation}
	In particular, it holds that
	\begin{equation}
		H_h(v_h) \to \widetilde{D}^2v \quad \text{in} \quad  [L^2(\Omega)]^{2 \times 2} \qquad \text{as}\quad  h \to 0^+.
	\end{equation}
\end{lemma}
\begin{proof}
	We begin by noting that $[\mathcal I_h^k v] = [v] = 0$ along $\Sigma$ because $v\in H^1(\Omega)$. This implies that $v_h=\mathcal I_h^k v$ is globally continuous and thus in $S_h^k$. The rest of the proof consists in two steps. First we prove the convergence of the broken Hessian, then we prove the convergence of the lifting terms. For this we need the following two estimates for curved elements $T\in\mathcal{T}_h$:
	\begin{equation}\label{eq:interp1}
		||D^2 \mathcal{I}_h^{k} w ||_{L^2(T)} \lesssim |w|_{H^2(T)} \text{\ } \forall w \in H^2(T)
	\end{equation}
	and for $1\le k'\le k$ and $0 \leq m \leq k'+1$,
	\begin{equation}\label{eq:interp2}
		||w - \mathcal{I}_h^k w||_{H^m(T)} \lesssim h_T^{k'+1-m}|w|_{H^{k'+1}(T)} \text{\ } \forall w \in H^{k'+1}(T). 
	\end{equation}
	The proofs of these estimates are standard for triangles and somewhat complicated for quadrilaterals. The proof for quadrilaterals can be found in \cite{DG_largebending_BNN_2021} for the case $k'=k$, but can be extended for $1 \leq k' \leq k$.  
	
	For the convergence of the broken Hessian, as in the proof of Lemma~\ref{lemma:weak_conv_hessian} we argue on each subdomain $\Omega_i$, $i=1,2$.  Let $v_i^{\epsilon} \in C^{\infty}(\Omega_i)$ be a smooth approximation of $v|_{\Omega_i} \in H^2(\Omega_i)$ such that $v^{\epsilon}_i \to v|_{\Omega_i}$ in $H^2(\Omega_i)$ as $\epsilon\to 0^+$. Also let $v_{i, h}^{\epsilon} := \mathcal{I}_h^k v_i^{\epsilon}$  be the Lagrange interpolant onto $S_h^k|_{\Omega_i}$. Notice that the traces of $v_{i,h}^\epsilon$ for $i=1,2$ do not necessarily match on the crease $\Sigma$, but this does not affect the local  argument provided below. Then for each $T \in \mathcal{T}_h^{i}$,
	\begin{align*}
		||D^2 v_h - D^2 v||_{L^2(T)} &\leq ||D^2 v_h - D^2 v_{i, h}^{\epsilon}||_{L^2(T)} + ||D^2 v_{i,h}^{\epsilon} - D^2v_i^{\epsilon}||_{L^2(T)} + ||D^2 v_i^{\epsilon} - D^2 v ||_{L^2(T)} \\
		&\lesssim |v-v_i^{\epsilon}|_{H^2(T)} + h|v_i^{\epsilon}|_{H^3(T)}
	\end{align*}
	from (\ref{eq:interp1}) and (\ref{eq:interp2}) with $m=2$ and $k'=2$. Summing over all $T \in \mathcal{T}_h^i$ gives
	$$||D_h^2 v_h - D^2 v||_{L^2(\Omega_i)} \leq C\Big( |v-v_i^{\epsilon}|_{H^2(\Omega_i)} + h|v_i^{\epsilon}|_{H^3(\Omega_i)} \Big),$$
	where $C$ is a constant independent of $\varepsilon$ and $h$.
	Now for any $\eta > 0$, we choose $\epsilon$ small enough so that $C|v-v_i^{\epsilon}|_{H^2(\Omega_i)} \leq \eta/2$, and then $h$ small enough so that $Ch|v^{\epsilon}|_{H^3(\Omega_i)} \leq \eta/2$. With these choices we have
	$$||D_h^2v_h - D^2 v||_{L^2(\Omega_i)} \leq \eta.$$
	Adding the results for $\Omega_1$ and $\Omega_2$ proves the strong convergence of $D_h^2 v_h$ to $\widetilde{D}^2 v$ in $[L^2(\Omega)]^{2 \times 2}$.
	
	Now we show that $R_h([\nabla v_h]) \to 0$ in $[L^2(\Omega)]^{2 \times 2}$. We use the stability estimate \eqref{eq:lift_stability} for $R_h$, together with the fact that $[\nabla v]|_e = 0$ for all $e \in \mathcal{E}_h^0\setminus \mathcal{E}_h^{\Sigma}$ to write
	$$||R_h([\nabla v_h])||_{L^2(\Omega)} \lesssim ||\h^{-1/2}[\nabla(v_h-v)]||_{L^2(\Gamma_h^0 \setminus \Sigma)}.$$
	Whence, the scaled trace estimate (\ref{eq:scaled_trace}), the (local) projection property $\mathcal{I}_h^kv_h|_T= v_h|_T$ and interpolation estimate (\ref{eq:interp2}) with $m=1$ and $k'=1$ yield for any $e \in \mathcal{E}_h^0\setminus \mathcal{E}_h^{\Sigma}$
	\begin{align*}
		||\h^{-1/2}[\nabla(v_h-v)]||_{L^2(e)}^2 &\lesssim \sum_{T \in \omega_e}\left[h_T^{-2}||\nabla(v_h-v)||_{L^2(T)}^2 + ||D_h^2(v_h-v)||_{L^2(T)}^2\right] \\
		&=\sum_{T \in \omega_e}\left[h_T^{-2}||\nabla(v_h - v - \mathcal{I}_h^k (v_h -v))||_{L^2(T)}^2 + ||D_h^2(v_h-v)||_{L^2(T)}^2\right] \\
		&\lesssim \sum_{T \in \omega_e} \|D^2v_h - D^2v\|_{L^2(T)}^2,
	\end{align*}
	where the hidden constant in the first inequality follows from the shape-regularity of the mesh ($\frac{h_T}{h_e}\le \frac{h_T}{\rho_T}\lesssim 1$).
	
	Summing over all the edges in $e \in \mathcal{E}_h^0\setminus \mathcal{E}_h^{\Sigma}$ and invoking the shape-regularity of $\{ \mathcal{T}_h \}_{h>0}$ gives the desired result:
	$$||R_h ([\nabla v_h])||_{L^2(\Omega)} \lesssim \Big( \sum_{T \in \mathcal{T}_h} \|D^2v_h - D^2v\|_{L^2(T)}^2 \Big)^{1/2} = \|D_h^2v_h-\widetilde{D}^2v\|_{L^2(\Omega)} \to 0 \qquad \textrm{as} \quad h \to 0^+.$$ 
	
	The convergences of $D_h^2 v_h$ and $R_h([\nabla v_h])$ directly imply the strong convergence of the discrete Hessian $H_h(v_h) = D^2_hv_h - R_h([\nabla v_h ])$.
\end{proof}

\subsection{Discrete Energy}
Using the finite element space and the discrete Hessian from above, we introduce the discrete energy
\begin{equation} \label{def:Eh}
	E_h({\bf y}_h) := E_h^S({\bf y}_h) + \theta^2 E_h^B({\bf y}_h),
\end{equation}
where
\begin{equation}
	E_h^S({\bf y}_h) := \frac{1}{8}\int_{\Omega} \Big( 2\mu |{\bf g}^{-1/2}( \nabla {\bf y}_h^T \nabla {\bf y}_h-{\bf g}) {\bf g}^{-1/2}|^2 + \lambda \text{tr}({\bf g}^{-1/2}(\nabla {\bf y}_h^T \nabla {\bf y}_h - {\bf g}) {\bf g}^{-1/2})^2 \Big) 
\end{equation}
and
\begin{equation} 
	\begin{split}
		E_h^B({\bf y}_h) :=& \frac{1}{24} \int_{\Omega} \Big( 2\mu |{\bf g}^{-1/2}H_h({\bf y}_h){\bf g}^{-1/2}|^2 + \frac{2\mu \lambda}{2\mu + \lambda} \text{tr}({\bf g}^{-1/2}H_h({\bf y}_h){\bf g}^{-1/2})^2 \Big) \\
		&+ \frac{\gamma}{2}||\h^{-1/2}[\nabla {\bf y}_h] ||_{L^2(\Gamma^0_h  \setminus \Sigma)}^2,
	\end{split}
\end{equation}
with $\gamma > 0.$

Compared to the exact energy \eqref{eq:energy_final} with the modification \eqref{eqn:Ebend}, the discrete energies above are defined for continuous piecewise polynomial deformations ${\bf y}_h$. In particular, the Hessian term in the bending energy \eqref{eqn:Ebend} is replaced by the discrete Hessian introduced above. 
This modification alone would not guarantee the uniform coercivity of the energy on $[S_h^k]^3$ with respect to the discrete seminorm $|\cdot|_{H^2_h(\Omega)}$. The purpose of the stabilization term with stabilization parameter $\gamma>0$ in the discrete bending energy $E_h^B$ is to restore the desired coercivity property. To see this, recall the following equivalence property provided by Lemma 2.6 of \cite{Prestrain_theoretical_BGNY}, together with the fact that $[v_h]=0$ for $v_h\in S_h^k$: 
\begin{equation}\label{eq:H2_equiv}
	C(\gamma)|v_h|_{H_h^2(\Omega)}^2 \leq ||H_h(v_h)||_{L^2(\Omega)}^2 + \frac{\gamma}{2}||\h^{-1/2}[\nabla v_h]||_{L^2(\Gamma_h^0 \setminus \Sigma)}^2 \lesssim |v_h|_{H_h^2(\Omega)}^2,
\end{equation}
where $C(\gamma)$ is a positive constant only depending on $\gamma$ and such that $C(\gamma) \to 0^+$ as $\gamma \to 0^+$. The coercivity of $E_h$ in $|.|_{H^2_h(\Omega)}$ is established below.

\begin{thm}[Partial coercivity of $E_h$] \label{thm:coercivity_prestrain}
	For all ${\bf y}_h \in [S_h^k]^3$, there holds
	\begin{equation} \label{eqn:bound_grad_ES}
		||\nabla {\bf y}_h||_{L^2(\Omega)}^2 \lesssim  | \Omega|^{1/2} \| {\bf g} \|_{L^\infty(\Omega)} E_h^S({\bf y}_h)^{1/2} + ||{\bf g}||_{L^1(\Omega)}.
	\end{equation}
	Moreover, for $\gamma > 0$, we have
	\begin{equation} \label{eqn:coercivity}
		|{\bf y}_h|_{H^2_h (\Omega)}^{2} \lesssim ||{\bf g}||_{L^{\infty}(\Omega)} E_h^B({\bf y}_h),
	\end{equation}
	where the hidden constant tends to $\infty$ as $\gamma \to 0^+$.
	In particular, we have for all ${\bf y}_h \in [S_h^k]^3$,
	\begin{equation}\label{eq:partial_coercivity}
		||\nabla {\bf y}_h||_{L^2(\Omega)}^{4} +  \theta^2 |{\bf y}_h|_{H^2_h (\Omega)}^{2} \lesssim E_h({\bf y}_h) +C,
	\end{equation}
	where $C$ is a constant depending only on $g$ and $\Omega$.
\end{thm}

\begin{proof}
	We start with the estimate on the gradient.
	For any $T \in \mathcal{T}_h$, we compute
	\begin{equation*}
		\frac{1}{2}\Big( \int_T |\nabla {\bf y}_h|^2 \Big)^2 \leq \sum_{i=1}^2 \Big( \int_T |\partial_i {\bf y}_h|^2 \Big)^2 \leq \sum_{i,j=1}^2 \Big(\int_T \partial_i {\bf y}_h \cdot\partial_j {\bf y}_h \Big)^2 =\left|\int_T \nabla {\bf y}_h^T \nabla {\bf y}_h \right|^2,
	\end{equation*}
	which implies that
	\begin{equation*}
		2^{-1/2} ||\nabla {\bf y}_h||_{L^2(\Omega)}^2 \leq \sum_{T \in \mathcal{T}_h} \int_T |\nabla {\bf y}_h^T \nabla {\bf y}_h | \leq \sum_{T \in \mathcal{T}_h} \int_T |\nabla {\bf y}_h^T \nabla {\bf y}_h - {\bf g}| + ||{\bf g}||_{L^1(\Omega)}.
	\end{equation*}
	Because $g$ is positive definite we have $| \cdot | \leq | {\bf g} | |{\bf g}^{-\frac{1}{2}}\cdot {\bf g}^{-\frac{1}{2}}|$.
	This together with a Cauchy-Schwarz inequality yields the desired estimate \eqref{eqn:bound_grad_ES}.
	
	To derive \eqref{eqn:coercivity}, we use again that $g$ is symmetric and positive definite to write
	$$ ||H_h({\bf y}_h)||_{L^2(\Omega)}^2 \lesssim ||{\bf g}||_{L^{\infty}(\Omega)} \int_{\Omega}|{\bf g}^{-1/2} H_h({\bf y}_h) {\bf g}^{-1/2}|^2$$
	and thus, thanks to the equivalence relation  \eqref{eq:H2_equiv},
	$$ |{\bf y}_h|^{2}_{H^2_h(\Omega)} \lesssim ||H_h({\bf y}_h)||_{L^2(\Omega)}^2 + \frac{\gamma}{2}||\h^{-1/2}[\nabla {\bf y}_h]||_{L^2(\Gamma_h^0  \setminus \Sigma)}^{2} \lesssim E_h^B({\bf y}_h)$$
	as desired.
	
	Estimate \eqref{eq:partial_coercivity} follows upon combining \eqref{eqn:bound_grad_ES} and \eqref{eqn:coercivity}.
\end{proof}

\section{$\Gamma$-Convergence of the Discrete Energy} \label{sec:gamma_conv}

The goal of this section is to establish the convergence of discrete minimizers of $E_h$ to continuous minimizers of $E$. This is done by essentially proving the $\Gamma$-convergence of $E_h$ to $E$.  The latter framework consists in establishing a lim-inf and lim-sup property, which are the focus of this section. These properties were proven in \cite{Prestrain_theoretical_BGNY} for the bending energy with a metric constraint, and we will follow similar reasoning in the preasymptotic case. First, however, we state the convergence of minimizers.

\begin{thm} \label{thm:Gamma_conv}
	Let $\{{\bf y}_h \}_{h>0} \subset [S_h^k]^3$ be a sequence of functions such that $E_h({\bf y}_h) \le \Lambda$  for $\Lambda$ independent of $h$. Assume furthermore that ${\bf y}_h$ is an almost global minimizer of $E_h$, meaning that 
	$$
	E_h({\bf y}_h) \leq \inf_{{\bf v}_h \in [S_h^k]^3} E_h({\bf v}_h) + \epsilon,$$ 
	where $\epsilon \to 0^+$ as $h \to 0^+$. 
	Then $\{ \bar{\bf y}_h \}_{h >0}$ with $\bar{\bf y}_h := {\bf y}_h - \strokedint_{\Omega} {\bf y}_h$ is precompact in $[L^2(\Omega)]^3$, and every cluster point $\bar{\bf y}$ of $\bar{\bf y}_h$ belongs to $[H^2(\Omega \setminus \Sigma)\cap H^1(\Omega)]^3$ and is a global minimizer of $E$, i.e.,
	$$
	E(\bar{\bf y}) = \inf_{{\bf v}\in [H^2(\Omega \setminus \Sigma)\cap H^1(\Omega)]^3} E({\bf v}).
	$$
	Moreover, up to a subsequence, there holds
	$$\lim_{h \to 0^+} E_h({\bf y}_h) = E(\bar{\bf y}).$$
\end{thm}

The proof of Theorem~\ref{thm:Gamma_conv} is standard (see \cite{DG_largebending_BNN_2021}) and follows from the lim-inf (Theorem~\ref{thm:liminf}) and lim-sup (Theorem~\ref{thm:limsup}) conditions  below, together with the compactness result of Lemma~\ref{lemma:compactness}. We note that the above theorem incorporates the possibility of a sequence of almost-minimizers rather than exact minimizers. This is to account for the fact that the algorithm proposed below (see Section~\ref{sec:gradient_flow}) uses a stopping criteria and thus does not necessarily reach the targeted minimizers.

We now establish the lim-inf and lim-sup conditions in our context, as they differ from the existing literature.

We begin with the lim-inf property.
\begin{thm}[Lim-inf of $E_h$] \label{thm:liminf}
	Let $\{ {\bf y}_h\}_{h>0} \subset [S_h^k]^3$ be a sequence of functions such that $E_h({\bf y}_h) \lesssim 1$. Then there exists $\bar{\bf y} \in [H^2(\Omega \setminus \Sigma)\cap H^1(\Omega)]^3$ with $\int_{\Omega} \bar{\bf y} = 0$ such that (up to a subsequence) the shifted sequence ${\bf \bar{y}}_h = {\bf y}_h - \strokedint_{\Omega} {\bf y}_h  \in [S_h^k]^3$ satisfies ${\bf \bar{y}}_h \to \bar{\bf y}$ in $[H^1(\Omega)]^3$ as $h \to 0^+$, and 
	\begin{equation}\label{e:liminf_prop}
		E(\bar{\bf y}) \leq \liminf_{h \to 0^+} E_h(\bar{\bf y}_h).
	\end{equation}
\end{thm}
\begin{proof}
	First, we use the partial coercivity \eqref{eq:partial_coercivity} of $E_h$ provided by Theorem \ref{thm:coercivity_prestrain}, along with the boundedness assumption $E_h({\bf y}_h) \lesssim 1$ to get 
	\begin{equation}\label{eq:boundedness}
		||\nabla {\bf y}_h||_{L^2(\Omega)}^2 + |{\bf y}_h|_{H_h^2(\Omega)}^2  \lesssim \max (1, \theta^{-2}),
	\end{equation}
	applying \eqref{eq:partial_coercivity} individually to each of the terms on the left-hand side due to the difference in scaling of the $L^2(\Omega)$ norms.
	We then apply the compactness result (Lemma \ref{lemma:compactness}) componentwise with $v_h =  y_{h,m}$ to guarantee the existence of $\bar{\bf y} \in [H^2(\Omega\setminus \Sigma)\cap H^1(\Omega)]^3$ with mean value zero such that $\bar{\bf y}_h \to \bar{\bf y}$ in $[L^2(\Omega)]^3$ and $\nabla \bar{\bf y}_h \to \nabla \bar{\bf y}$ in $[L^2(\Omega)]^{3 \times 2}$. 
	
	Next, we establish the lim-inf relation  by showing the weak convergence of the integrands. Since we only have weak convergence of the discrete Hessian on the subdomains $\Omega_1$ and $\Omega_2$, we show the lim-inf relation on each subdomain, and then use the property 
	\begin{equation} \label{eq:liminfprop}
		\liminf_{h \to 0^+} E_h^1(\bar{\bf y}_h) + \liminf_{h \to 0^+} E_h^2(\bar{\bf y}_h) \leq \liminf_{h \to 0^+} (E_h^1(\bar{\bf y}_h)+E_h^2(\bar{\bf y}_h)), 
	\end{equation}
	where $E_h^i$ denotes the energy computed on the subdomain $\Omega_i$, $i=1,2$.
	
	Thanks to the weak convergence of the reconstructed Hessian (Lemma \ref{lemma:weak_conv_hessian}), we find that on each subdomain, $H_h(\bar{\bf y}_h) \rightharpoonup D^2 \bar{\bf y}$, which implies that ${\bf g}^{-1/2}H_h(\bar{\bf y}_h){\bf g}^{-1/2} \rightharpoonup {\bf g}^{-1/2}D^2 \bar{\bf y} {\bf g}^{-1/2}$ as $h \to 0^+$ on each subdomain. 
	In view of the weak lower semi-continuity of the $L^2(\Omega)$ norm, we get 
	\begin{equation} \label{eq:hessian_inf}
		\int_{\Omega_i} |{\bf g}^{-1/2} D^2 \bar{\bf y} {\bf g}^{-1/2}|^2 \leq \liminf_{h \to 0^+} \int_{\Omega_i} |{\bf g}^{-1/2} H_h(\bar{\bf y}_h){\bf g}^{-1/2}|^2 \qquad  i=1,2.
	\end{equation} 
	
	Similarly, the linearity of the trace term implies that
	$$\text{tr}({\bf g}^{-1/2}H_h(\bar{\bf y}_h){\bf g}^{-1/2}) \rightharpoonup \text{tr}({\bf g}^{-1/2}D^2 \bar{\bf y} {\bf g}^{-1/2}) \qquad \text{as}\quad  h \to 0^+$$
	on each subdomain, which implies that
	\begin{equation} \label{eq:trace_hess_inf}
		\int_{\Omega_i} \text{tr}({\bf g}^{-1/2} D^2 \bar{\bf y} {\bf g}^{-1/2})^2 \leq \liminf_{h \to 0^+} \int_{\Omega_i} \text{tr}({\bf g}^{-1/2}H_h(\bar{\bf y}_h) {\bf g}^{-1/2})^2 \qquad  i=1,2.
	\end{equation}
	
	For the nonlinear stretching term, we let  ${\boldsymbol \phi} \in [L^2(\Omega_i)]^{2 \times 2}$ and compute
	\begin{align*}
		&\int_{\Omega_i} \boldsymbol{\phi} : (\nabla \bar{\bf y}_h^T \nabla \bar{\bf y}_h - \nabla \bar{\bf y}^T \nabla \bar{\bf y}) =  \int_{\Omega_i} \boldsymbol{\phi}: [(\nabla \bar{\bf y}_h - \nabla \bar{\bf y})^T( \nabla \bar{\bf y}_h - \nabla \bar{\bf y})] \\
		&\qquad \qquad \qquad + \int_{\Omega} \boldsymbol{\phi} : [(\nabla \bar{\bf y}_h-\nabla \bar{\bf y})^T\nabla \bar{\bf y}] +\int_{\Omega_i} \boldsymbol{\phi}: [\nabla \bar{\bf y}^T (\nabla \bar{\bf y}_h - \nabla \bar{\bf y})] \\
		& \qquad \leq ||\boldsymbol{\phi}||_{L^2(\Omega_i)}||\nabla \bar{\bf y}_h - \nabla \bar{\bf y}||_{L^4(\Omega_i)}^2 +2||\boldsymbol{\phi}||_{L^2(\Omega_i)}||\nabla \bar{\bf y}_h - \nabla \bar{\bf y}||_{L^4(\Omega_i)} ||\nabla \bar{\bf y}||_{L^4(\Omega_i)},
	\end{align*}
	where we used the identity
	\begin{equation} \label{eq:identity_matrices}
		{\bf A}^T{\bf A}-{\bf B}^T{\bf B} = ({\bf A}-{\bf B})^T({\bf A}-{\bf B})+({\bf A}-{\bf B})^T{\bf B} + {\bf B}^T({\bf A}-{\bf B}), \quad {\bf A},{\bf B}\in\mathbb{R}^{3\times 2}.   
	\end{equation}
	
	Note that the boundedness (\ref{eq:boundedness}) implies that $\nabla \bar{\bf y}_h \to \nabla \bar{\bf y}$ strongly in $[L^4(\Omega_i)]^{3 \times 2}$. Thus,
	$$||\nabla \bar{\bf y}_h - \nabla \bar{\bf y}||_{L^4(\Omega_i)} \to 0 \qquad \textrm{as}\quad h\to 0^+,$$
	which proves that $\nabla \bar{\bf y}_h^T \nabla \bar{\bf y}_h  \rightharpoonup \nabla \bar{\bf y}^T \nabla \bar{\bf y}$ in $[L^2(\Omega_i)]^{2 \times 2}$ as $h\to 0^+$.
	Thanks to the weak lower semi-continuity of the $L^2(\Omega_i)$ norm, we deduce that for $i=1,2$
	\begin{equation} \label{eq:nonlinear_inf}
		\int_{\Omega_i} |{\bf g}^{-1/2}(\nabla \bar{\bf y}^T \nabla \bar{\bf y}- {\bf g}){\bf g}^{-1/2}|^2 \leq \liminf_{h \to 0^+} \int_{\Omega_i} |{\bf g}^{-1/2}( \nabla \bar{\bf y}_h^T \nabla \bar{\bf y}_h - {\bf g}) {\bf g}^{-1/2}|^2.
	\end{equation}
	Similarly, for the trace term,
	\begin{equation} \label{eq:trace_nonlin_inf}
		\int_{\Omega_i} \text{tr} ({\bf g}^{-1/2} (\nabla \bar{\bf y}^T \nabla \bar{\bf y}-{\bf g}) {\bf g}^{-1/2})^2 \leq \liminf_{h \to 0^+} \int_{\Omega_i} \text{tr}({\bf g}^{-1/2} (\nabla \bar{\bf y}_h^T \nabla \bar{\bf y}_h-{\bf g}) {\bf g}^{-1/2})^2.
	\end{equation}
	
	The desired lim-inf property \eqref{e:liminf_prop} follows from relations \eqref{eq:hessian_inf}, \eqref{eq:trace_hess_inf}, \eqref{eq:nonlinear_inf},  \eqref{eq:trace_nonlin_inf}, \eqref{eq:liminfprop}, and the positivity of the jump terms in $E_h(\bar{\bf y}_h)$.
	
\end{proof}

We now give a proof of the lim-sup property. Because of the strong convergence of the discrete Hessian, we actually get strong convergence of the recovery sequence.

\begin{thm}[Lim-sup of $E_h$]\label{thm:limsup}
	For any ${\bf y} \in [H^2(\Omega \setminus \Sigma)\cap H^1(\Omega)]^3$, there exists $\{ {\bf y}_h \}_{h>0} \subset [S_h^k]^3$ such that ${\bf y}_h \to {\bf y}$ in $[H^1(\Omega)]^3$ as $h \to 0^+$ and $E({\bf y}) = \lim_{h \to 0^+} E_h({\bf y}_h)$.
\end{thm}

\begin{proof}
	We show that the sequence ${\bf y}_h$ of piecewise Lagrange interpolants (as defined in Lemma \ref{lemma:strong_conv_hessian}) satisfies the desired properties. First, the interpolation estimates \eqref{eq:interp2} with $m=1$ and $k'=1$ give directly that ${\bf y}_h \to {\bf y}$ in $[H^1(\Omega)]^3$.
	
	Next, we establish convergence of the bending and stretching parts separately.
	The convergence of the bending term,
	\begin{equation}\label{eq:limsup_B}
		\begin{split}
			\lim_{h \to 0^+} \bigl(\int_{\Omega} |{\bf g}^{-1/2} H_h({\bf y}_h) {\bf g}^{-1/2}|^2 &+ \int_\Omega \text{tr}({\bf g}^{-1/2} H_h({\bf y}_h) {\bf g}^{-1/2})^2+\frac{\gamma}{2}||\h^{-1/2}[\nabla {\bf y}_h]||^2_{L^2(\Gamma_h^0  \setminus \Sigma})\bigr) \\
			&= \int_{\Omega} |{\bf g}^{-1/2}  \widetilde D^2 {\bf y} {\bf g}^{-1/2}|^2 + \text{tr}({\bf g}^{-1/2}  \widetilde D^2{\bf y} {\bf g}^{-1/2})^2, 
		\end{split}
	\end{equation}
	follows from the strong convergence of the reconstructed Hessian and the property of the jump term established in Lemma~\ref{lemma:strong_conv_hessian}.
	
	For the stretching term, we use again the identity \eqref{eq:identity_matrices}, so that a Cauchy-Schwarz inequality yields
	\begin{equation*}
		||\nabla {\bf y}_h^T \nabla {\bf y}_h - \nabla {\bf y}^T \nabla {\bf y}||_{L^2(\Omega)} \leq ||\nabla {\bf y}_h - \nabla {\bf y}||_{L^4(\Omega)}^2 + 2||\nabla {\bf y}_h - \nabla {\bf y}||_{L^4(\Omega)} ||\nabla {\bf y}||_{L^4(\Omega)}. 
	\end{equation*}
	The convergence $\nabla {\bf y}_h \to \nabla {\bf y}$ in $[L^4(\Omega)]^{3\times 2}$ as $h \to 0^+$ follows from the already established $H^1(\Omega)$-convergence of the same quantities along with the continuous embedding $H^1(\Omega) \subset L^4(\Omega)$. From this we infer that 
	$$\lim_{h \to 0^+} (\nabla {\bf y}_h^T \nabla {\bf y}_h - {\bf g}) =  (\nabla {\bf y}^T \nabla {\bf y} - {\bf g})$$
	and thus
	\begin{equation}\label{e:limsup_S}
		\begin{split}
			\lim_{h \to 0^+} \int_{\Omega} \bigl( & |{\bf g}^{-1/2}(\nabla {\bf y}_h^T \nabla {\bf y}_h - {\bf g}){\bf g}^{-1/2}|^2 +  \text{tr}({\bf g}^{-1/2} (\nabla {\bf y}_h^T \nabla {\bf y}_h - {\bf g}){\bf g}^{-1/2})^2\bigr) \\
			&= \int_{\Omega} |{\bf g}^{-1/2}(\nabla {\bf y}^T \nabla {\bf y} - {\bf g}){\bf g}^{-1/2}|^2 + \text{tr}({\bf g}^{-1/2} (\nabla {\bf y}^T \nabla {\bf y} - {\bf g}) {\bf g}^{-1/2})^2.
		\end{split}
	\end{equation}
	
	Relations \eqref{eq:limsup_B} and \eqref{e:limsup_S} yield the desired convergence of $E_h({\bf y}_h)$ to $E({\bf y})$ and end the proof.
\end{proof}

\section{Discrete Energy Minimization} \label{sec:gradient_flow}

Having established the convergence of global minimizers in the previous section (Theorem~\ref{thm:Gamma_conv}), we now present an energy-decreasing gradient flow for the discrete energy. As we shall see, this gradient flow is guaranteed to locate critical points of the discrete energy. 

Ideally, given an initial deformation ${\bf y}_h^0 \in [S^k_h]^3$ and a pseudo-timestep $\tau>0$, successive approximations are found to minimize 
\begin{equation*}
	[S_h^k]^3 \ni {\bf v}_h  \mapsto \frac{1}{2\tau}||{\bf v}_h-{\bf y}_h^n||_{H_h^2(\Omega)}^2 + E_h({\bf v}_h),
\end{equation*}
where the discrete norm $||\cdot||_{H^2_h(\Omega)}$ is defined in (\ref{def:discrete_H2_norm}).
Because the stretching part is nonconvex, a semi-implicit algorithm is proposed in \cite[Section 6]{Prestrain_BGNY_2022} and proved to be (conditionally) energy-decreasing. In this work, we propose a fully explicit algorithm in the stretching part, which considerably simplifies the algorithm while retaining the energy decreasing property. 

Given ${\bf y}_h^0 \in [S_h^k]^3$ and $\tau>0$, we define 
${\bf y}_h^{n+1}\in[S_h^k]^3$ to satisfy
\begin{equation} \label{eq:gf_original}
	\tau^{-1}({\bf y}_h^{n+1}-{\bf y}_h^{n}, {\bf v}_h)_{H_h^2(\Omega)} + a_h^S({\bf y}_h^n, {\bf v}_h) + \theta^2 a_h^B({\bf y}_h^{n+1}, {\bf v}_h) = 0, \ \ \ \forall {\bf v}_h \in [S_h^k]^3,
\end{equation}
where 
\begin{align*}
	a_h^B({\bf y}_h, {\bf v}_h) &:= \frac{\mu}{6} \int_{\Omega} \Big( {\bf g}^{-1/2}H_h({\bf y}_h) {\bf g}^{-1/2} \Big) : \Big( {\bf g}^{-1/2}H_h({\bf v}_h) {\bf g}^{-1/2} \Big) \\ &+ \frac{\mu\lambda}{6(2\mu+\lambda)} \int_{\Omega} \text{tr}\Big( {\bf g}^{-1/2}H_h({\bf y}_h){\bf g}^{-1/2} \Big) \text{tr} \Big( {\bf g}^{-1/2}H_h({\bf v}_h) {\bf g}^{-1/2} \Big) \\ &+ \gamma(\h^{-1}[\nabla {\bf y}_h], [\nabla {\bf v}_h])_{L^2(\Gamma_h^0\setminus \Sigma)} 
\end{align*}
and (with a slight abuse of notation) $a_h^S({\bf y}_h, {\bf v}_h) := a_h^S({\bf y}_h; {\bf y}_h,{\bf v}_h)$ with
\begin{align*}
	a_h^S({\bf z}_h; {\bf y}_h,{\bf v}_h) & :=
	\frac{\mu}{2}\int_{\Omega} \Big({\bf g}^{-1/2} (\nabla {\bf v}_h^T \nabla {\bf y}_h + \nabla {\bf y}_h^T \nabla {\bf v}_h){\bf g}^{-1/2} \Big) : \Big( {\bf g}^{-1/2}(\nabla {\bf z}_h^T \nabla{\bf z}_h - {\bf g}){\bf g}^{-1/2} \Big ) \\ &+ \frac{\lambda}{4} \int_{\Omega} \text{tr} \Big({\bf g}^{-1/2} (\nabla {\bf v}_h^T \nabla {\bf y}_h + \nabla {\bf y}_h^T \nabla {\bf v}_h){\bf g}^{-1/2} \Big) \text{tr} \Big( {\bf g}^{-1/2}(\nabla {\bf z}_h^T \nabla{\bf z}_h - {\bf g}){\bf g}^{-1/2} \Big ).
\end{align*}
In \eqref{eq:gf_original}, the forms $a_h^B({\bf y}_h, {\bf v}_h)$ and $a_h^S({\bf y}_h^n, {\bf v}_h)$ are the G\^ateau derivative of $E_h^B$ and $E_h^S$, respectively, at ${\bf y}_h$ in the direction ${\bf v}_h$.

It will be useful for us to rewrite (\ref{eq:gf_original}) using the variation $\delta {\bf y}_h^{n+1}:= {\bf y}_h^{n+1} - {\bf y}_h^n$, i.e., 
\begin{equation} \label{eq:gf}
	\tau^{-1}(\delta {\bf y}_h^{n+1}, {\bf v}_h)_{H_h^2(\Omega)}+ \theta^2 a_h^B(\delta {\bf y}_h^{n+1}, {\bf v}_h) = -a_h^S({\bf y}_h^n, {\bf v}_h) - \theta^2 a_h^B({\bf y}_h^n, {\bf v}_h).
\end{equation}

The existence and uniqueness of $\delta {\bf y}_h^{n+1}$ satisfying the linear problem \eqref{eq:gf} directly follows from the coercivity property 
\begin{equation} \label{eq:coercivity}
	\tau^{-1} ||{\bf v}_h||_{H_h^2(\Omega)}^{2} + 2\theta^2 E_h^B({\bf v}_h) = \tau^{-1} ({\bf v}_h, {\bf v}_h)_{H_h^2(\Omega)}+ \theta^2 a_h^B({\bf v}_h, {\bf v}_h) \ \forall {\bf v}_h \in [S_h^k]^3.
\end{equation}
A stability estimate for one step of the gradient flow is provided in Lemma~\ref{lemma:solvability_gf}  while Corollary~\ref{cor:energy_decay} guarantees an energy decay property.

We emphasize again that the explicit treatment of the stretching term offers a computational advantage since the system matrix (and its preconditioner) do not need to be reconstructed at each iteration. In turn, this drastically reduces the running time of the experiments provided in Section 6.

\begin{lemma} [Stability of one step of the gradient flow] \label{lemma:solvability_gf} The unique solution $\delta {\bf y}_h^{n+1} \in [S^k_h]^3$ to (\ref{eq:gf}) satisfies
	\begin{equation}\label{eq:stab}
		\tau^{-1}  \| \delta {\bf y}_h^{n+1} \|_{H^2_h(\Omega)}^2 + 2\theta^2 E_h^B(\delta {\bf y}_h^{n+1}) \lesssim \tau E_h^S({\bf y}_h^n) \| \nabla {\bf y}_h^{n} \|_{L^4(\Omega)}^2+ \frac{\theta^2}{8}E_h^B({\bf y}_h^n).
	\end{equation}
	
\end{lemma}

\begin{proof}
	On the one hand, we have already noted that the bilinear form $\tau^{-1}(\cdot, \cdot)_{H_h^2(\Omega)} + \theta^2 a_h^B(\cdot, \cdot)$ is coercive, see \eqref{eq:coercivity}. 
	On the other hand, given ${\bf y}_h^n \in [S_h^k]^3$, the discrete and continuous Cauchy-Schwarz inequalities imply that the mapping
	$$
	{\bf v}_h \mapsto F({\bf v}_h):=-a_h^S({\bf y}_h^n, {\bf v}_h) - \theta^2 a_h^B({\bf y}_h^n, {\bf v}_h) 
	$$
	satisfies for all ${\bf v}_h \in [S_h^k]^3$
	\begin{equation*} 
		\begin{split}
			F({\bf v}_h) &\lesssim  E_h^{S}({\bf y}_h^n)^{1/2}||\nabla {\bf y}_h^n||_{L^4(\Omega)} ||\nabla {\bf v}_h||_{L^4(\Omega)} + \theta^2 E_h^B({\bf y}_h^n)^{1/2} E_h^B({\bf v}_h)^{1/2}
			\\
			& \lesssim E_h^{S}({\bf y}_h^n)^{1/2}||\nabla {\bf y}_h^n||_{L^4(\Omega)} || {\bf v}_h||_{H^2_h(\Omega)} + \theta^2 E_h^B({\bf y}_h^n)^{1/2} E_h^B({\bf v}_h)^{1/2}
			\\
			& \lesssim \left( \tau E_h^S({\bf y}_h^n)\| \nabla {\bf y}_h^n \|_{L^4(\Omega)}^2  + \frac{\theta^2}{8} E_h^B({\bf y}_h^n) \right)^{1/2}
			\left( \tau^{-1} \| {\bf v}_h \|_{H^2_h(\Omega)}^2 + 2\theta^2 E_h^B({\bf v}_h) \right)^{1/2},
		\end{split}
	\end{equation*}
	where we invoked the discrete Sobolev inequality provided by Lemma~\ref{lemma:disc_sobolev} to justify the second inequality.
	
	Consequently, choosing ${\bf v}_h = \delta {\bf y}_h^{n+1}$ in \eqref{eq:gf} yields \eqref{eq:stab}. 
\end{proof}

We now proceed to the main result of this section and prove that the gradient flow is energy-decreasing. We do this in two steps. First we establish that the energy decreases at each step of the gradient flow, provided that $\tau$ satisfies a condition that depends on the energy at the current step. Then, by induction, we establish a condition on $\tau$ that is actually uniform over all timesteps. 

\begin{prop} [Energy decay for the gradient flow] \label{prop:energy_decay}
	Let ${\bf y}_h^n \in [S_h^k]^3$ and define
	\begin{equation}
		\label{eq:d}
		d_{h, \tau} ({\bf y}_h^n) :=  (1+\tau^2 E_h({\bf y}_h^n))  \big(h_{\min}^{-1}(1+E_h({\bf y}_h^n)^{1/2})\big)+\tau  E_h({\bf y}_h^n)+E_h({\bf y}_h^n)^{1/2}\big)>0,
	\end{equation}
	where $h_{\min}:=\min_{T\in\mathcal{T}_h}h_T$.
	There exists a constant $C>0$ independent of $h$, $n$ and $\theta$ such that if
	\begin{equation}\label{eq:tau_cond2}
		\tau \leq  \frac 1{C d_{h, \tau}({\bf y}_h^n)}
	\end{equation}
	then
	\begin{equation} \label{eq:energy_decay}
		E_h({\bf y}_h^{n+1}) + \frac{1}{2\tau}||\delta {\bf y}_h^{n+1}||_{H_h^2(\Omega)}^2 \leq E_h({\bf y}_h^n).
	\end{equation}
\end{prop}
\begin{proof}
	We take ${\bf v}_h = \delta {\bf y}_h^{n+1} ={\bf y}_h^{n+1}-{\bf y}_h^n$ in (\ref{eq:gf_original}) to get
	\begin{equation} \label{eqn:GF_delta}
		\tau^{-1}||\delta {\bf y}_h^{n+1}||_{H_h^2(\Omega)}^{2} + a_h^S({\bf y}_h^n; {\bf y}_h^{n}, \delta {\bf y}_h^{n+1}) + \theta^2 a_h^B({\bf y}_h^{n+1}, \delta {\bf y}_h^{n+1})=0.
	\end{equation}
	Now we proceed in three steps. 
	
	\noindent
	{\em Step 1: Energy relation.} Since $a_h^{S}({\bf y}_h^n; \cdot, \cdot)$ is bilinear and symmetric, we have 
	\begin{equation*}
		a_h^S({\bf y}_h^n; {\bf y}_h^n, \delta {\bf y}_h^{n+1}) = \frac{1}{2}a_h^S({\bf y}_h^n; {\bf y}_h^{n+1}, {\bf y}_h^{n+1}) - \frac{1}{2}a_h^S({\bf y}_h^n; {\bf y}_h^n, {\bf y}_h^n) - \frac{1}{2}a_h^S({\bf y}_h^n; \delta {\bf y}_h^{n+1}, \delta {\bf y}_h^{n+1}).
	\end{equation*}
	
	Furthermore, using the identity $(a-b)b = \frac{1}{2}a^2-\frac{1}{2}b^2-\frac{1}{2}(a-b)^2$, we get
	\begin{equation*}
		\begin{split}
			\frac{1}{2}a_h^S({\bf y}_h^n; {\bf y}_h^{n+1}, {\bf y}_h^{n+1}) - \frac{1}{2}a_h^S({\bf y}_h^n; {\bf y}_h^n, {\bf y}_h^n) &= \int_\Omega W_h^n:((\nabla {\bf y}_h^n)^T(\nabla {\bf y}_h^n)-{\bf g}) \\
			&= E_h^{S}({\bf y}_h^{n+1})-E_h^{S}({\bf y}_h^n)-\frac{1}{2}||W_h^n||_{L^2(\Omega)}^2,
		\end{split}
	\end{equation*}
	where 
	\begin{equation*}
		W_h^n := (\nabla {\bf y}_h^{n+1})^T \nabla {\bf y}_h^{n+1} - (\nabla {\bf y}_h^n)^T \nabla {\bf y}_h^n.
	\end{equation*}
	Therefore, we find that
	\begin{equation*}
		a_h^S({\bf y}_h^n; {\bf y}_h^{n+1}, \delta {\bf y}_h^{n+1}) = E_h^{S}({\bf y}_h^{n+1})-E_h^{S}({\bf y}_h^n)- \frac{1}{2}a_h^S({\bf y}_h^n; \delta {\bf y}_h^{n+1}, \delta {\bf y}_h^{n+1})-\frac{1}{2}||W_h^n||_{L^2(\Omega)}^2.
	\end{equation*}
	Similarly, for $E_h^B$, we have
	\begin{equation*}
		a_h^B({\bf y}_h^{n+1},\delta {\bf y}_h^{n+1}) = E_h^B({\bf y}_h^{n+1}) -  E_h^B({\bf y}_h^{n}) + \frac{1}{2}a_h^B(\delta {\bf y}_h^{n+1}, \delta {\bf y}_h^{n+1}) \geq E_h^B({\bf y}_h^{n+1}) - E_h^B({\bf y}_h^n).
	\end{equation*}
	Plugging these last two equations into \eqref{eqn:GF_delta} gives  
	\begin{equation} \label{eq:energy_diff_bound}
		E_h({\bf y}_h^{n+1}) - E_h({\bf y}_h^n)+\tau^{-1}||\delta {\bf y}_h^{n+1}||_{H_h^2(\Omega)}^2 \leq R_h^n,
	\end{equation}
	where
	\begin{equation*}
		R_h^n := \frac{1}{2}||W_h^n||_{L^2(\Omega)}^2 + \frac{1}{2} a_h^S({\bf y}_h^n; \delta {\bf y}_h^{n+1}, 
		\delta {\bf y}_h^{n+1}).
	\end{equation*}
	
	\noindent
	{\em Step 2: Bound for $R_h^n$.}  Now we prove that 
	\begin{equation}\label{eq:R_bound}
		|R_h^n| \leq \frac C 2 d_{h, \tau}({\bf y}_h^n)||\delta {\bf y}_h^{n+1}||_{H_h^2(\Omega)}^2,
	\end{equation}
	where $d_{h, \tau}({\bf y}_h^n)$ is defined by (\ref{eq:d}) and $C$ is a constant independent of $h$, $n$ and $\theta$.
	
	Using the relation \eqref{eq:identity_matrices}, we write $W_h^n$ as 
	\begin{equation*}
		W_h^n = (\nabla \delta {\bf y}_h^{n+1})^T \nabla {\bf y}_h^n + (\nabla {\bf y}_h^{n})^T \nabla \delta {\bf y}_h^{n+1} + (\nabla \delta {\bf y}_h^{n+1})^T \nabla \delta {\bf y}_h^{n+1},
	\end{equation*}
	and from the discrete Sobolev inequality (Lemma \ref{lemma:disc_sobolev}), we get
	\begin{equation*} \label{eq:W_bound}
		||W_h^n||_{L^2(\Omega)}^2 \lesssim \big(||\nabla {\bf y}_h^n||_{L^4(\Omega)}^2+||\delta {\bf y}_h^{n+1}||_{H_h^2(\Omega)}^2 \big) ||\delta {\bf y}_h^{n+1}||_{H_h^2(\Omega)}^2.
	\end{equation*}
	The stability estimate from Lemma~\ref{lemma:solvability_gf} implies 
	\begin{equation*}\label{eq:H2_bound}
		||\delta {\bf y}_h^{n+1}||_{H_h^2(\Omega)}^2 \lesssim \tau \Big(\tau E_h^S({\bf y}_h^n)||\nabla {\bf y}_h^n||_{L^4(\Omega)}^2+\theta^2 E_h^B({\bf y}_h^n)\Big)
	\end{equation*}
	and so
	\begin{equation} \label{eq:W_bound_bis}
		||W_h^n||_{L^2(\Omega)}^2 \lesssim \big((1+\tau^2 E_h^S({\bf y}_h^n) )||\nabla {\bf y}_h^n||_{L^4(\Omega)}^2+ \tau \theta^2 E_h^B({\bf y}_h^n) \big) ||\delta {\bf y}_h^{n+1}||_{H_h^2(\Omega)}^2.
	\end{equation}
	
	To estimate $||\nabla {\bf y}_h^n||_{L^4(\Omega)}$ in the above expression, we recall that the gradient estimate \eqref{eqn:bound_grad_ES} provides a control on $||\nabla {\bf y}_h^n||_{L^2(\Omega)}$. Whence, together with an inverse inequality we have 
	\begin{equation} \label{eq:L4_bound}
		||\nabla {\bf y}_h^n||_{L^4(\Omega)}^2 \lesssim h_{\min}^{-1} ||\nabla {\bf y}_h^n||_{L^2(\Omega)}^2 \lesssim h_{\min}^{-1} \left(1+E_h^S({\bf y}_h^n)^{1/2}\right).
	\end{equation}
	
	Plugging (\ref{eq:L4_bound}) into (\ref{eq:W_bound_bis}), we get
	\begin{equation*}\label{eq:W_bound2}
		||W_h^n||_{L^2(\Omega)}^2 \lesssim \left\lbrace (1+\tau^2 E_h^S({\bf y}_h^n))  \big(h_{\min}^{-1}(1+E_h^S({\bf y}_h^n)^{1/2})\big)+\tau \theta^2 E_h^B({\bf y}_h^n)\right\rbrace ||\delta {\bf y}_h^{n+1}||_{H_h^2(\Omega)}^2.
	\end{equation*}
	To finish the estimate for $R_h^n$, it remains to bound $a_h^S({\bf y}_h^n; \delta {\bf y}_h^n, \delta {\bf y}_h^n)$. To this end, we apply the Cauchy-Schwarz inequality and the discrete Sobolev inequality (Lemma \ref{lemma:disc_sobolev}) to obtain
	\begin{equation*} 
		|a_h^S({\bf y}_h^n; \delta {\bf y}_h^{n+1}, \delta {\bf y}_h^{n+1})| \lesssim E_h^{S}({\bf y}_h^n)^{1/2}||\delta {\bf y}_h^{n+1}||^2_{H_h^2(\Omega)}.
	\end{equation*}
	Combining the last two estimates and recalling that $E_h({ \bf y}_{h}^n) = E_h^S({ \bf y}_{h}^n)+ \theta^2 E_h^B({ \bf y}_{h}^n)$ gives (\ref{eq:R_bound}). 
	
	\noindent
	{\em Step 3: Conditional energy decay.} Inserting \eqref{eq:R_bound} in \eqref{eq:energy_diff_bound} and rearranging the terms, we arrive at
	\begin{equation}
		E_h({\bf y}_h^{n+1})+\left(\frac 1 \tau-\frac{C}{2}d_{h, \tau} ({\bf y}_h^n)\right)||\delta {\bf y}_h^{n+1}||_{H_h^2(\Omega)}^2 \leq E_h({\bf y}_h^n),
	\end{equation}
	and so the energy decay property \eqref{eq:energy_decay} is proven as long as \eqref{eq:tau_cond2} holds.
\end{proof}

We now derive a uniform energy decay property by showing that the condition \eqref{eq:tau_cond2} holds uniformly in $n$.

\begin{cor} [Uniform energy decay] \label{cor:energy_decay}
	Let ${\bf y}_h^0 \in [S_h^k]^3$, and assume that 
	\begin{equation}\label{e:time_step_cond}
		\tau \leq \frac{1}{C d_{h, \tau}({\bf y}_h^0)},
	\end{equation}
	where $C$ is the constant in Proposition~\ref{prop:energy_decay}. The sequence of successive iterates ${\bf y}_h^{n+1}$, where ${\bf y}_h^{n+1}$ satisfies \eqref{eq:gf_original}, is well-defined. Furthermore, we have
	\begin{equation} \label{eq:uniform_est}
		E_h({\bf y}_h^{N+1}) + \frac{1}{2\tau} \sum_{n=0}^N ||\delta {\bf y}_h^{n+1}||_{H_h^2(\Omega)}^2 \leq E_h({\bf y}_h^0) \qquad  \forall N \geq 0.
	\end{equation}
\end{cor}

\begin{proof}
	We show \eqref{eq:uniform_est}, along with the property $d_{h, \tau}({\bf y}_h^{N+1}) \leq d_{h,\tau}({\bf y}_h^{0})$, by induction on $N \geq 0$. 
	For $N=0$, $\tau \leq \frac{1}{Cd_{h, \tau}({\bf y}_h^0)}$ by assumption, so Proposition \ref{prop:energy_decay} guarantees that 
	$$
	E_h({\bf y}_h^1) + \frac{1}{2\tau} \| \delta {\bf y}_h^1 \|_{H^2_h(\Omega)}^2 \leq E_h({\bf y}_h^0).
	$$
	This is \eqref{eq:uniform_est} for $N=0$. Furthermore, the above inequality also implies that $E_h({\bf y}_h^1) \leq E_h({\bf y}_h^0)$, and thus $d_{h, \tau}({\bf y}_h^1) \leq d_{h, \tau}({\bf y}_h^0)$ in view of the definition \eqref{eq:d} of $d_{h, \tau}$.

	We now assume that \eqref{eq:uniform_est} holds for $N$; that is, we assume that
	\begin{equation} \label{eq:recursion_assumption}
		E_h({\bf y}_h^{N}) + \frac{1}{2\tau} \sum_{n=0}^{N-1} ||\delta {\bf y}_h^{n+1}||_{H_h^2(\Omega)}^2 \leq E_h({\bf y}_h^0),
	\end{equation}
	and also that $d_{h, \tau}({\bf y}_h^N) \leq d_{h, \tau}({\bf y}_h^0)$, and we prove the relations for $N+1$. Because $d_{h, \tau}({\bf y}_h^{N}) \leq d_{h, \tau}({\bf y}_h^0)$, we have $\tau \leq \frac 1{C d_{h, \tau}({\bf y}_h^{0})} \leq \frac 1{C d_{h, \tau}({\bf y}_h^{N})}$. Therefore, Proposition~\ref{prop:energy_decay} guarantees that 
	$$
	E_h({\bf y}_h^{N+1}) + \frac{1}{2\tau} \| \delta {\bf y}_h^{N+1} \|_{H^2_h(\Omega)}^2 \leq E_h({\bf y}_h^N).
	$$
	This, together with \eqref{eq:recursion_assumption}, yields \eqref{eq:uniform_est} for $N+1$.
	The relation \eqref{eq:uniform_est} implies that $E_h({\bf y}_h^{N+1}) \leq E_h({\bf y}_h^0)$, and thus
	$d_{h, \tau}({\bf y}_h^{N+1}) \leq d_{h, \tau}({\bf y}_h^0)$.
\end{proof}

\begin{remark}[Gradient flow timestep condition]
	The condition on the gradient flow timestep \eqref{e:time_step_cond} depends on the mesh size $h_{\min}$ but is rather mild. In fact, for $0<\tau\leq 1$, it reads
	$$
	\tau \leq C h_{\min},
	$$
	where $C$ is a constant only depending on $E_h({\bf y}_h^0)$.
\end{remark}

The final result of this section shows that the iterates obtained via (\ref{eq:gf_original}) converge to a critical point of $E_h$. In order to prove it, we require the following conservation of averages result, whose proof is identical to that of Proposition 5.3 in \cite{Prestrain_theoretical_BGNY}.

\begin{lemma} [Conservation of averages] \label{lemma:avg_cons}
	Let ${\bf y}_h^0 \in [S_h^k]^3$. Then all the iterates ${\bf y}_h^n$, $n \geq 1$, of the gradient flow \eqref{eq:gf_original} satisfy 
	$$\int_{\Omega} {\bf y}_h^n = \int_{\Omega} {\bf y}_h^0.$$
\end{lemma}

Now we can show that the iterates of \eqref{eq:gf_original} converge.

\begin{prop} [Limit of the gradient flow]
	Fix $h>0$ and let ${\bf y}_h^0 \in [S_h^k]^3$. Also assume that $E({\bf y}_h^0) \lesssim 1$, and that $\strokedint {\bf y}_h^0 = 0$. Let $\{ {\bf y}_h^n \}_{n}$ be the sequence produced by the discrete gradient flow (\ref{eq:gf_original}).  Then there exists ${\bf y}_h^{\infty}$ such that (up to a subsequence) ${\bf y}_h^n \to {\bf y}_h^{\infty}$ as $n \to \infty$ and ${\bf y}_h^{\infty}$ is a critical point of $E_h$.
\end{prop}

\begin{proof}
	From Corollary~\ref{cor:energy_decay}, the sequence $\{ E_h({\bf y}_h^n) \}_{n}$ is bounded, and from Lemma \ref{lemma:avg_cons}, $\int_{\Omega} {\bf y}_h^n = \int_{\Omega} {\bf y}_h^0 = 0$. Combining the Poincar{\'e}-Friedrichs inequality \eqref{eq:disc_poincare_1} from Lemma \ref{lemma:disc_poincare} with Theorem \ref{thm:coercivity_prestrain}, we conclude that $||{\bf y}_h^n||_{L^2(\Omega)}$ is bounded. Thus, since $[S_h^k]^3$ is finite-dimensional, $\{ {\bf y}_h^n \}_{n}$ converges strongly (up to a subsequence not relabeled) in $[L^2(\Omega)]^3$ to some ${\bf y}_h^{\infty} \in [S_h^k]^3$ as $n\rightarrow\infty$. The fact that $[S_h^k]^3$ is of finite dimension implies that we in fact have ${\bf y}_h^n \to {\bf y}_h^{\infty}$ in any norm on $[S_h^k]^3$; in particular, ${\bf y}_h^n \to {\bf y}_h^{\infty}$ in $[H_h^2(\Omega)]^3$.
	
	Provided that
	\begin{equation} \label{eq:bending_limit}
		\lim_{n \to \infty} a_h^B({\bf y}_h^n, {\bf v}_h) = a_h^B({\bf y}_h^{\infty}, {\bf v}_h)
	\end{equation}
	and 
	\begin{equation} \label{eq:stretching_limit}
		\lim_{n \to \infty} a_h^S({\bf y}_h^n; {\bf y}_h^n, {\bf v}_h) = a_h^S({\bf y}_h^{\infty}; {\bf y}_h^{\infty}, {\bf v}_h) 
	\end{equation}
	for all ${\bf v}_h \in [S_h^k]^3$,
	taking the limit as $n \to \infty$ of (\ref{eq:gf_original}) gives 
	\begin{equation}
		0 = a_h^S({\bf y}_h^{\infty}; {\bf y}_h^{\infty}, {\bf v}_h) + \theta^2 a_h^B({\bf y}_h^{\infty},{\bf v}_h).
	\end{equation}
	But 
	\begin{equation}
		a_h^S({\bf y}_h^{\infty}; {\bf y}_h^{\infty}, {\bf v}_h) + \theta^2 a_h^B({\bf y}_h^{\infty},{\bf v}_h) = \delta E({\bf y}_h^{\infty})({\bf v}_h) = 0,
	\end{equation}
	which implies that ${\bf y}_h^{\infty}$ is a critical point as desired.
	
	Now we establish (\ref{eq:bending_limit}) and (\ref{eq:stretching_limit}). The equation (\ref{eq:bending_limit}) can be readily shown using the Cauchy-Schwarz inequality and the fact that $||{\bf y}_h^n - {\bf y}_h^{\infty} ||_{H_h^2(\Omega)} \to 0$ as $n \to \infty$. To prove (\ref{eq:stretching_limit}), we rewrite 
	\begin{multline*}
		|a_h^S({\bf y}_h^n; {\bf y}_h^n, {\bf v}_h)-a_h^S({\bf y}_h^{\infty}; {\bf y}_h^{\infty}, {\bf v}_h)| \lesssim \\
		\Big| \int_{\Omega} ({\bf g}^{-1/2}(\nabla {\bf v}_h^T \nabla ({\bf y}_h^n-{\bf y}_h^{\infty}) + \nabla ({\bf y}_h^n -{\bf y}_h^{\infty})^T \nabla {\bf v}_h ){\bf g}^{-1/2}):({\bf g}^{-1/2}((\nabla {\bf y}_h^{\infty})^T\nabla {\bf y}_h^{\infty} -{\bf g}){\bf g}^{-1/2})\Big| \\ 
		+ \Big| \int_{\Omega} ({\bf g}^{-1/2}(\nabla {\bf v}_h^T \nabla {\bf y}_h^n + (\nabla {\bf y}_h^n)^T \nabla {\bf v}_h){\bf g}^{-1/2}):({\bf g}^{-1/2}((\nabla {\bf y}_h^n)^T \nabla {\bf y}_h^n-(\nabla {\bf y}_h^{\infty})^T \nabla {\bf y}_h^{\infty}){\bf g}^{-1/2}) \Big| \\
		\\ + \Big| \int_{\Omega} \text{tr}({\bf g}^{-1/2}(\nabla {\bf v}_h^T \nabla ({\bf y}_h^n-{\bf y}_h^{\infty}) + \nabla ({\bf y}_h^n -{\bf y}_h^{\infty})^T \nabla {\bf v}_h ){\bf g}^{-1/2})\text{tr}({\bf g}^{1/2}((\nabla {\bf y}_h^{\infty})^T\nabla {\bf y}_h^{\infty}-{\bf g}){\bf g}^{-1/2})\Big| \\ 
		+ \Big| \int_{\Omega} \text{tr}({\bf g}^{-1/2}(\nabla {\bf v}_h^T \nabla {\bf y}_h^n + (\nabla {\bf y}_h^n)^T \nabla {\bf v}_h){\bf g}^{-1/2})\text{tr}({\bf g}^{-1/2}((\nabla {\bf y}_h^n)^T \nabla {\bf y}_h^n-(\nabla {\bf y}_h^{\infty})^T \nabla {\bf y}_h^{\infty}){\bf g}^{-1/2}) \Big| \\
		=: |T_1| + |T_2| + |T_3| + |T_4|
	\end{multline*}
	
	The leading constants are not included because they do not affect the computation. 
	
	We focus on the first two terms, $T_1$ and $T_2$, because the reasoning for the trace terms is similar. For $T_1$, 
	Cauchy-Schwarz, the definition of $E_h^{S}$, and Lemma \ref{lemma:disc_sobolev} give 
	\begin{align*}
		|T_1|
		&\lesssim ||\nabla {\bf v}_h||_{L^4(\Omega)} ||\nabla ({\bf y}_h^n-{\bf y}_h^{\infty})||_{L^4(\Omega)} E_h^{S}({\bf y}_h^{\infty})^{1/2}\\
		&\lesssim ||{\bf v}_h||_{H_h^2(\Omega)} ||{\bf y}_h^n-{\bf y}_h^{\infty}||_{H_h^2(\Omega)} E_h^{S}({\bf y}_h^{\infty})^{1/2}.
	\end{align*}
	But $||{\bf y}_h^n-{\bf y}_h^{\infty}||_{H_h^2(\Omega)} \to 0$, and $E_h^{S}({\bf y}_h^{\infty})$ is finite by Corollary \ref{cor:energy_decay}. So $|T_1| \to 0$ as $n \to \infty$.
	
	Similarly, for $T_2$,
	\begin{align*}
		|T_2|
		&\lesssim||\nabla {\bf v}_h||_{L^4(\Omega)} || \nabla {\bf y}_h^n||_{L^4(\Omega)} ||(\nabla {\bf y}_h^{\infty})^T \nabla {\bf y}_h^{\infty} - (\nabla {\bf y}_h^n)^T \nabla {\bf y}_h^n||_{L^2(\Omega)} \\
		&\lesssim ||{\bf v}_h||_{H_h^2(\Omega)} ||{\bf y}_h^n||_{H_h^2(\Omega)} ||(\nabla {\bf y}_h^{\infty})^T \nabla {\bf y}_h^{\infty} - (\nabla {\bf y}_h^n)^T \nabla {\bf y}_h^n||_{L^2(\Omega)}.
	\end{align*}
	Since $\{{\bf y}_h^n\}_{n}$ converges in $H_h^2(\Omega)$, $||{\bf y}_h^n||_{H_h^2}$ is bounded, and using the relation \eqref{eq:identity_matrices} we obtain
	\begin{align*}
		||(\nabla {\bf y}_h^{\infty})^T \nabla {\bf y}_h^{\infty} - (\nabla {\bf y}_h^n)^T \nabla {\bf y}_h^n||_{L^2(\Omega)} &\leq ||\nabla({\bf y}_h^n - {\bf y}_h^{\infty})^T \nabla({\bf y}_h^n - {\bf y}_h^{\infty})||_{L^2(\Omega)} \\ 
		&+||\nabla ({\bf y}_h^n - {\bf y}_h^{\infty})^T\nabla {\bf y}_h^{\infty}||_{L^2(\Omega)} + ||(\nabla {\bf y}_h^{\infty})^T \nabla ({\bf y}_h^n - {\bf y}_h^{\infty})||_{L^2(\Omega)} \\
		&\leq ||\nabla({\bf y}_h^n-{\bf y}_h^{\infty})||_{L^4(\Omega)}^2 + 2||\nabla {\bf y}_h^{\infty}||_{L^4(\Omega)}||\nabla({\bf y}_h^n - {\bf y}_h^{\infty})||_{L^4(\Omega)} \\
		&\lesssim ||{\bf y}_h^n-{\bf y}_h^{\infty}||_{H_h^2(\Omega)}^2 + 2||{\bf y}_h^{\infty}||_{H_h^2(\Omega)}||{\bf y}_h^n - {\bf y}_h^{\infty}||_{H_h^2(\Omega)} \to 0
	\end{align*}
	as $n \to \infty$. Thus $|T_2| \to 0$ as $n \to \infty$. Applying similar reasoning to the trace terms, we obtain (\ref{eq:stretching_limit}) and complete the proof.
\end{proof}

\section{Numerical Experiments}

We illustrate the convergence results from the previous section, along with the capabilities of the preasymptotic model, via numerical experiments. Unless otherwise specified, all experiments are run with $\mu=6$, $\lambda=8$, and pseudo-timestep $\tau=10^{-2}$. The polynomial degree for the deformation and the lifting is set to $k=2$, and we take the stabilization parameter $\gamma=1$. The gradient flow terminates when 
$$\tau^{-1} |E_h({\bf y}_h^{n+1})-E_h({\bf y}_h^n)| \leq tol,$$
where $tol$ will be set in each of the experiments. Finally, although the analysis above deals with the case $\theta>0$, we will also consider the case $\theta=0$ in the experiment of Section~\ref{sec:oscillation}, namely the minimization of the stretching energy only. Similar results as provided by Lemma~\ref{lemma:solvability_gf} and Proposition~\ref{prop:energy_decay} can be derived for $\theta=0$ following the arguments provided in \cite[Section 6]{Prestrain_theoretical_BGNY} (treating the case $\sigma_h=0$).

Note that, because our energy is nonconvex, we cannot guarantee that the deformations found by our gradient flow are global or even local minimizers. 
However, in some cases, such as the bubble experiment below, we have reason to believe that the deformation found by the gradient flow is a minimizer. We use this case to demonstrate the convergence results established in Section 4. In addition, it is difficult to design input data for which we can find an exact minimizer of (\ref{eq:energy_final}). Instead, we treat the solution of largest refinement as the exact solution when computing errors.  

\subsection{Disc with ``Bubble" metric}
We first consider the unit disc with the following metric, which corresponds to a bubble shape with positive Gaussian curvature: 
\begin{equation} \label{eqn:tensor_bubble}
	{\bf g}(x_1,x_2) = \left(\begin{array}{cc}
		1+\alpha \frac{\pi^2}{4}\cos(\frac{\pi}{2}(1-r))^2 \frac{x_1^2}{r^2} & \alpha \frac{\pi^2}{4}\cos(\frac{\pi}{2}(1-r))^2 \frac{x_1 x_2}{r^2} \\ \alpha\frac{\pi^2}{4}\cos(\frac{\pi}{2}(1-r))^2 \frac{x_1 x_2}{r^2} & 1+\alpha \frac{\pi^2}{4}\cos(\frac{\pi}{2}(1-r))^2 \frac{x_2^2}{r^2}
	\end{array}\right).
\end{equation}
In the expression above, $r = \sqrt{x_1^2 + x_2^2}$ and $\alpha = 0.2$. A compatible deformation for this metric is given by
\begin{equation}\label{eq:compatible_def}
	{\bf y}(x_1, x_2)= \Big(x_1, x_2, \sqrt{\alpha}\sin\big(\frac{\pi}{2}(1-r)\big)\Big)^T.
\end{equation}

Note that this deformation is a minimizer of (\ref{eq:energy_final}) in the case that $\theta=0$. Thus, we expect discrete minimizers to have similar shape when $\theta$ is small. Indeed, we find this to be true experimentally, and if $\theta$ is large, our algorithm converges to the flat configuration, which is also a local minimizer. No other deformations have been found by experiments. Thus, we have some confidence that the configuration reached by this experiment is an approximate minimizer. 

We set $tol=10^{-9}$ for the stopping criterion and let $\theta^2=10^{-5}$. The initial deformation is the continuous Lagrange interpolant of the shallow paraboloid $-\frac{x_1^2}{16}-\frac{x_2^2}{16}$.  This gives a slight initial bending that allows the simulation to find a non-flat minimizer. (The flat configuration is a local minimizer of the preasymptotic energy, so that a simulation that starts flat stays flat.) 

To demonstrate convergence of minimizers as $h \to 0^+$, we run the simulation on grids of increasing refinement. The grid of maximum refinement has $61,827$ degrees of freedom (5,120 quadrilaterals), and we take the reference solution ${\bf y}_h^{\text{ref}}$ to be the solution on this mesh. In Table \ref{table:bubble_energy_conv}, we record the energy of the final deformation for the different mesh sizes. The final energy decreases as the mesh is refined, with a value of (approximately) 3.9e-05 for the finest mesh, whereas the energy of the interpolant of the  compatible deformation \eqref{eq:compatible_def}, which is itself an almost-minimizer, on a grid of $246,531$ degrees of freedom (20,480 quadrilaterals) is 2.858e-05. Running more experiments with higher refinement should give us energies closer to this minimum; however the time costs of such experiments are currently prohibitive.

In Table \ref{table:bubble_sol_conv}, we show the convergence of the coarse solutions ${\bf y}_h^N$ to the reference solution ${\bf y}_h^{\text{ref}}$. Errors are computed in the $L^2$, $H^1$, and $L^{\infty}$ norm.  Note that to save memory in our simulations, we only write solutions every thousand timesteps. Thus, the iterate at which errors are computed is rounded up to the nearest thousand from the iterate at which the tolerance is reached. For the large number of time steps required in this experiment, the effect on our results is negligible.

We observe a decrease of errors to zero in all three norms as the number of degrees of freedom increases. In particular, the convergence of minimizers in the $L^2$ norm agrees with our result in Section 4. Note that the theory does not predict a rate of convergence; it only ensures that the error converges to $0$. By running more simulations with greater refinement, we may be able to observe a convergence rate experimentally.
\begin{center}
	\begin{table}[htbp!]
		\begin{tabular}{ |c|c| } 
			\hline
			Number of DoFs & Final Energy \\
			\hline
			1011 & 5.082e-05 \\
			\hline
			3939 &  4.294e-05 \\ 
			\hline
			15555 & 4.089e-05 \\
			\hline
			61827 & 3.936e-05 \\
			\hline
		\end{tabular}
		\caption{Energy $E_h$ at the final timestep of the gradient flow for increasing refinement levels of the mesh.}\label{table:bubble_energy_conv}
	\end{table}
\end{center}

\begin{center}
	\begin{table}[htbp!]
		\begin{tabular}{ |c|c|c|c| } 
			\hline
			Number of DoFs & $||{\bf y}_h^N - {\bf y}_h^{\text{ref}}||_{L^2(\Omega)}$ & $||{\bf y}_h^N - {\bf y}_h^{\text{ref}}||_{H^1(\Omega)}$ &  $||{\bf y}_h^N - {\bf y}_h^{\text{ref}}||_{L^{\infty}(\Omega)}$\\
			\hline
			1011 & 1.111e-03 & 2.138e-02  & 1.570e-03 \\
			\hline
			3939 & 2.856e-04 & 1.652e-02 & 3.867e-04 \\ 
			\hline
			15555 & 1.193e-04 & 7.268e-03  & 2.044e-04 \\
			\hline 
		\end{tabular}
		\caption{Errors between the final solution at each refinement level and the final solution on the reference mesh.}\label{table:bubble_sol_conv}
	\end{table}
\end{center}

To conclude this experiment, we comment on the effect of replacing the second fundamental form in the physically correct bending energy \eqref{eqn:EbendII} by the Hessian to get \eqref{eqn:Ebend}. This modification, also considered for instance in \cite{CM2008,Wrinkling_BK_2014}, greatly simplifies the design of a numerical scheme since the variational derivative of $E^B(\mathbf{y})$ is linear. It can be rigorously justified in specific cases, e.g. when $E^S(\mathbf{y})=0$ as shown in \cite[Proposition A.1]{Prestrain_theoretical_BGNY} (which is not limited to small deformations), and we expect this simplification to remain \emph{valid} for deformation with small stretching energy. To illustrate the effect numerically, we give in Figure~\ref{fig:comparison_II_D2y} the element-wise bending energy of the final deformation in the case $\mu=6$, $\lambda=0$ and a mesh with 3 refinements. We observe that the distribution of the local bending energies with the (discrete) second fundamental form and Hessian are similar. Note that this is not surprising for low energy deformations for which the stretching energy is expected to be small due to the different scaling $1$ versus $\theta^2$ for the two parts of the energy \cite{BKN2013}.

\begin{figure}[htbp!]
	\begin{center}
		\includegraphics[width=4.00cm]{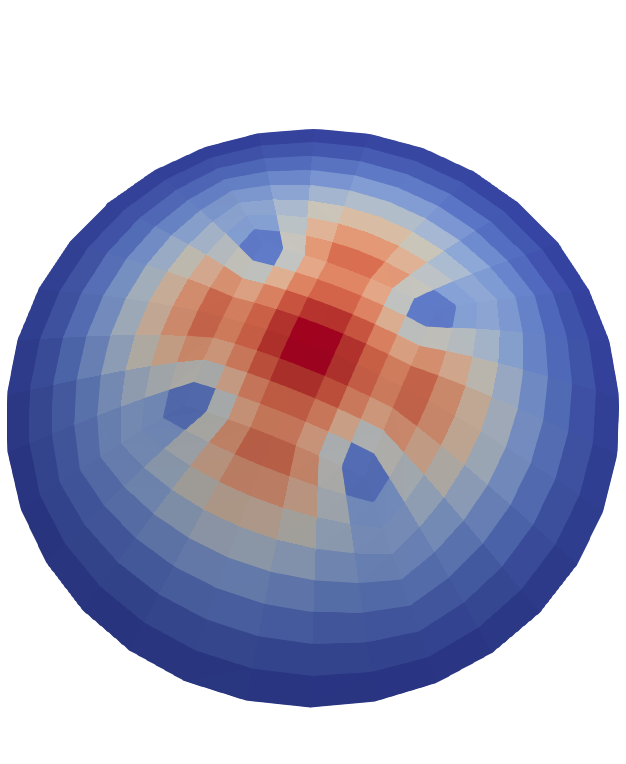}
		\hspace{0.8cm}
		\includegraphics[width=4.00cm]{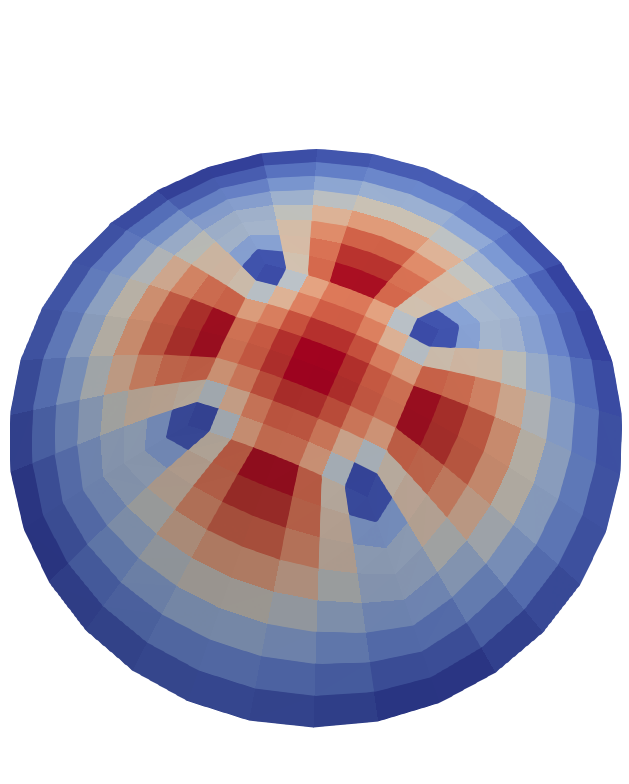}
	\end{center}
	\caption{Element-wise bending energy of the final deformation using the second fundamental form (left) and the Hessian (right). The same color map is used for both plot and the ranges of values are [2.53e-8,6.82e-8] and [3.69e-8,6.84e-8], respectively.}
	\label{fig:comparison_II_D2y}
\end{figure}

\subsection{Disc with Oscillating Boundary} \label{sec:oscillation}
The next experiment is inspired by \cite{KVS_2011,Review_LM_2021}, in which a hydrogel disc of negative Gaussian curvature is observed to develop more oscillations along the boundary as its thickness is reduced. It is included to demonstrate the capabilities of the preasymptotic model. A similar experiment involving discontinuous elements can be found in \cite{BGM_2022}. 

The material is prestrained according to the metric
$${\bf g}={\bf J}^T(\widetilde {\bf g} \circ \boldsymbol{\zeta}){\bf J},$$
where ${\bf J}$ is the Jacobian matrix for the change of variables $(r,\phi)=\boldsymbol{\zeta}(x_1,x_2)$ from Cartesian to polar coordinates, $\widetilde {\bf g}(r, \phi)$ is the first fundamental form of $\widetilde {\bf y}(\widetilde\Omega)$, where $\widetilde\Omega=\boldsymbol{\zeta}(\Omega)$, and
\begin{equation} \label{eqn:target_oscillating}
	\widetilde {\bf y}(r,\phi) = (r\cos(\phi), r\sin(\phi), 0.2r^4\sin(6\phi)).
\end{equation}
This deformation corresponds to a disk with six wrinkles.

For the discretization, our subdivision consists of $1,280$ quadrilaterals and a total of $15,555$ degrees of freedom. The initial deformation ${\bf y}_h^0$ is taken to be the continuous Lagrange interpolant of $\widetilde {\bf y}\circ \boldsymbol{\zeta}$ with $\widetilde {\bf y}$ given by (\ref{eqn:target_oscillating}). We set $tol=10^{-10}$ for the stopping criterion and run the experiment for $\theta^2 = 10^{-1}$, $10^{-2}$, $10^{-3}$, and $0$.

Because the gradient descent algorithm is slow to converge, we use a Nesterov-type acceleration to speed up the computation. This algorithm is analogous to the FISTA (see \cite{BT09, CD15}) for minimizing $F=f+g$ with $f=E^S$ (explicit) and $g=E^B$ (implicit), although $E^S$ is not convex as assumed in the listed references. The algorithm proceeds as follows (see \cite{BGM_2022} for the semi-implicit scheme):
Given ${\bf y}_h^0\in[S_h^k]^3$, set ${\bf w}_h^0={\bf y}_h^0$. Then for $n=0,1,\ldots$ do 
\begin{enumerate}
	\item Find ${\bf y}_h^{n+1} \in [S_h^k]^3$ satisfying (\ref{eq:gf_original}) with  ${\bf y}_h^n$ replaced by ${\bf w}_h^n$. 
	\item Set ${\bf w}_h^{n+1}={\bf y}_h^{n+1}+\eta_{n+1}({\bf y}_h^{n+1}-{\bf y}_h^{n})$, where $\eta_{n+1}=\frac{t_{n+1}-1}{t_{n+2}}$ and $\{t_n\}_{n\ge 1}$ satisfies
	\begin{equation} \label{def:tn_acc}
		t_1=1 \quad \mbox{and} \quad  \quad t_{n+1} = \sqrt{t_n^2+\frac{1}{4}}+\frac{1}{2} \quad \mbox{for } n=0,1,\ldots. 
	\end{equation}
\end{enumerate}

Details about the FISTA can be found in \cite{BT09, BGM_2022, CD15}. See also \cite{Nesterov1983} for the original algorithm proposed by Nesterov. 

The results are given in Figure \ref{fig:results_oscillating} and match those in \cite{BGM_2022}. As observed in the experiments of \cite{KVS_2011,Review_LM_2021}, the number of oscillations increases as the thickness decreases.

\begin{figure}[htbp!]
	\begin{center}
		\includegraphics[width=4.00cm]{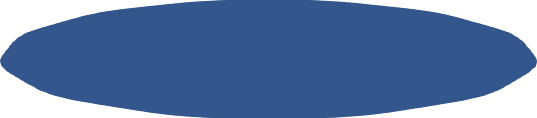}
		\hspace{1.5cm}
		\includegraphics[width=4.00cm]{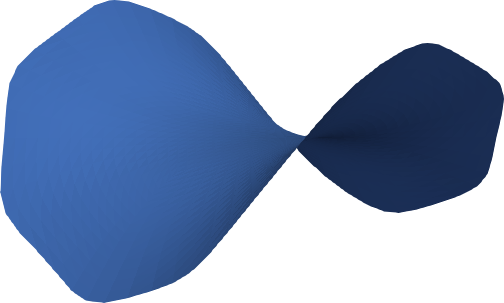} \\
		\vspace{1.3cm}
		\includegraphics[width=4.00cm]{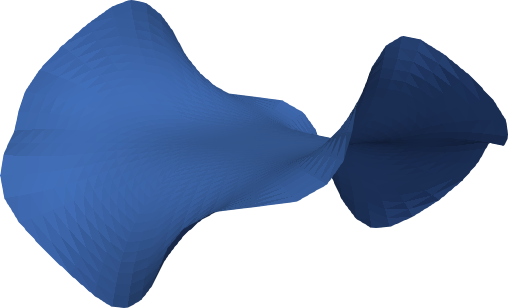}
		\hspace{1.5cm}
		\includegraphics[width=4.00cm]{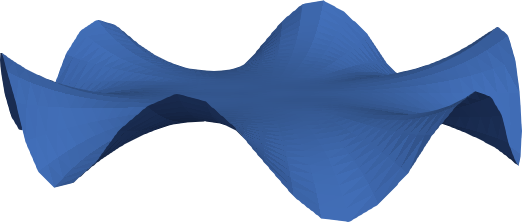}
	\end{center}
	\caption{Final configurations for the oscillating boundary experiment. Left to right and top to bottom: $\theta^2 = 10^{-1}$, $10^{-2}$, $10^{-3}$, and $0$.}
	\label{fig:results_oscillating}
\end{figure}

We note that the $\theta=10^{-3}$ case appears to be particularly sensitive to the minimization algorithm and the data of the experiment. For instance, if we linearize at ${\bf y}_h^n$ instead of ${\bf w}_h^n$ in the Nesterov-type algorithm, we obtain the result shown in Figure \ref{fig:results_oscillating_yn}. Starting from the almost-flat configuration from Section 6.1 instead of a disc with 6 oscillations produces a similar result. This configuration has a lower energy than that in Figure \ref{fig:results_oscillating}, and also matches more closely the results obtained in Figure 1 of \cite{KVS_2011}. Understanding the effect of various parameters on the final configuration is a topic for future work. 

\begin{figure}[htbp!]
	\begin{center}
		\includegraphics[width=4.00cm]{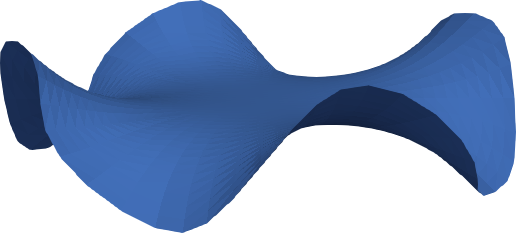}
	\end{center}
	\caption{Final configuration for the oscillating boundary experiment with $\theta = 10^{-3}$, where we use the linearization $a_h^S({\bf y}_h^n; \cdot, \cdot)$ instead of $a_h^S({\bf w}_h^n; \cdot, \cdot)$.}
	\label{fig:results_oscillating_yn}
\end{figure}

\subsection{Flapping Device}\label{s:flapping}
This experiment illustrates the advantages of the preasymptotic model when folding is present. Its aim is to simulate the flapping device discussed in the introduction.

Recall that in this experiment, the metric is ${\bf g} = {\bf I}_2$, namely the plate has no prestrain. Moreover, a time-dependent boundary condition is used to model the compression dynamic. More details are provided in \cite{BGM_2022}, but the implementation is summarized here. We introduce a physical timestep $\Delta t$ and consider the time points $t_m = m\Delta t$ for $m=1, \ldots, M$, where $M$ is a positive integer. We write $T=M\Delta t$ for the final time, and we use ${\bf y}_h^{(m)}:={\bf y}_h(t_m)$ to denote the output of the gradient flow \eqref{eq:gf_original} using boundary data at time $t_m$. Finally, given $t\in[0,T]$, let $\hat{\bf y}_h$ be the LDG approximation of the solution $\hat{\bf y}$ to the following bi-Laplacian problem used to incorporate the boundary conditions in the initial deformation:
\begin{equation} \label{eq:pb_BC}
	\left\{\begin{array}{rcll}
		\Delta^2\hat{\bf y} & = & \hat{\bf f} & \mbox{in } \Omega \\
		\nabla\hat{\bf y} & = & \Phi(t) & \mbox{on } \Gamma^D \\
		\hat{\bf y} & = & \boldsymbol{\varphi}(t) & \mbox{on } \Gamma^D\cup\Gamma^M \\
	\end{array}\right.
\end{equation}
supplemented with the following natural boundary conditions
\begin{equation} \label{eq:pb_BC2}
	\left\{
	\begin{array}{rcll}
		\quad (D^2\hat y_m){\bf n} & = & \mathbf{0} & \mbox{on } \partial\Omega\setminus\Gamma^D \\
		\nabla(\Delta \hat y_m)\cdot{\bf n} & = & 0 & \mbox{on } \partial\Omega\setminus(\Gamma^D\cup\Gamma^M)
	\end{array}\right.
\end{equation}
for $m=1,2,3$, where $\hat {\bf f}$ is a fictitious force, $\Gamma^D$ is the portion of the boundary on which we enforce Dirichlet boundary conditions and $\Gamma^M$ is the portion on which we enforce mixed boundary conditions. By mixed boundary conditions, we mean that only the value of the deformation (and not its gradient) is specified on $\Gamma^M$. We thus have the following time dependent algorithm:
\begin{enumerate}
	\item \textbf{Initialization} ($t=0$):
	Obtain $\hat {\bf y}_h^{(0)}$ by solving \eqref{eq:pb_BC}-\eqref{eq:pb_BC2} with the prescribed conditions evaluated at $t=0$. Set ${\bf y}_h^{(0)} = \hat {\bf y}_h^{(0)}$;
	
	\item  \textbf{Dynamics} ($t\in(0,T]$): for $m=1,2,\ldots,M$, do
	\begin{enumerate}
		\item Obtain $\delta \hat {\bf y}_h^{(m)}$ by solving \eqref{eq:pb_BC}-\eqref{eq:pb_BC2} with the increment boundary conditions
		\begin{equation*} 
			\begin{split}
				\nabla \delta \hat {\bf y}_h^{(m)} &= \Phi(t_m)-\Phi(t_{m-1}), \quad \mbox{on } \Gamma^D \\
				\delta \hat {\bf y}_h^{(m)} &= \boldsymbol{\varphi}(t_m)-\boldsymbol{\varphi}(t_{m-1}), \quad \mbox{on } \Gamma^D\cup\Gamma^M.  \\
			\end{split}
		\end{equation*}
		
		\item Starting from $\hat {\bf y}_h^{(m)}:={\bf y}_h^{(m-1)}+\delta \hat {\bf y}_h^{(m)}$, obtain ${\bf y}_h^{(m)}$ using the  gradient flow \eqref{eq:gf_original} to minimize the discrete energy with the data evaluated at $t=t_m$
	\end{enumerate}
\end{enumerate}
Notice that the \emph{Dynamics} step solves for the deformation increments in order to take advantage of the previous time computation. Otherwise, at each timestep, the gradient flow would restart from the solution to \eqref{eq:pb_BC}-\eqref{eq:pb_BC2}.

In the case of the flapping device, the domain is the rectangle $(0, 15) \times (0, 9.6)$ with a crease along the center, namely $\Sigma = (0,15)\times \{x_2=4.8\}$. We let $\Gamma_{\text{left}}=\{x_1=0\} \times (0,9.6)$ and $\Gamma_{\text{right}}=\{x_1=15\} \times (0,9.6)$. For the initialization step ($t=0$), we set $\Gamma^M = \Gamma_{\text{left}}$ and $\Gamma^D = \Gamma_{\text{right}}$. Moreover, the fictitious force is given by
\begin{equation*}
	\hat{\bf f}(x_1,x_2) =  \begin{cases} 
		\mathbf{0} & 0 \leq x_1 \leq 10 \\
		(0, 0, 0.002)^T & 10 < x_1 \leq 15, \\
	\end{cases}
\end{equation*}
and the boundary functions $\boldsymbol{\varphi}$ and $\Phi$ are defined as follows:
\begin{equation*}
	\boldsymbol{\varphi}(x_1,x_2) = \begin{cases} 
		(x_1, x_2, 0)^T & x_1=0 \\
		(10, x_2, 5)^T & x_1=15, \\
	\end{cases}
\end{equation*}

\begin{equation*}
	\Phi(x_1,x_2) = 
	\begin{pmatrix} 0 & 0 & 1 \\ 0& 1 & 0  \end{pmatrix}^T, \quad x_1=15.
\end{equation*}

For the part (a) of the dynamics step, we set $\Gamma^M=\Gamma_{\text{left}} \cup \Gamma_{\text{right}}$ and $\Gamma^D=\emptyset$, and we define the increment function $\boldsymbol{\varphi}$, for which we compute $\boldsymbol{\varphi}(t_m)-\boldsymbol{\varphi}(t_m)$, on each side as follows. 
On $\Gamma_{\text{left}}$, $\boldsymbol{\varphi}$ is given by
\begin{equation} \label{eqn:BC_dynamics}
	\restriction{\boldsymbol{\varphi}}{\Gamma_{\text{left}}} = \begin{cases}
		\Big(x_1, 9.6t_m(1-\frac{x_2}{4.8}), \frac{x_2}{4.8}h(t_m)\Big)^T & 0 \leq x_2 < 4.8 \\
		\Big(x_1, 9.6t_m(1-\frac{x_2}{4.8}), (2-\frac{x_2}{4.8})h(t_m)\Big)^T & 4.8 \leq x_2 \leq 9.6,
	\end{cases}
\end{equation}
where 
\begin{equation*}
	h(t)=\sqrt{t(9.6-t)}.
\end{equation*}
This boundary condition forces the deformation to take a tent shape on the left side of the computational domain. The coefficients of the height function $h(t)$ on the upper and lower halfs range from $0$ to $1$ ($0$ on the ends, $1$ in the center), ensuring that the sheet has maximum height in the center and decreases linearly to $0$ on the top and bottom corners. 

On $\Gamma_{\text{right}}$ we impose
\begin{equation*}
	\restriction{\boldsymbol{\varphi}}{\Gamma_{\text{right}}}(t_m,x_1,x_2) = \Big(x_1, 9.6(1-\frac{x_2}{4.8}), h(t_m)\Big)^T, \quad (x_1,x_2)\in\Gamma_{\text{right}},
\end{equation*}
where $h(t)$ represents the maximum height of the sheet at time $t$ and is given by 
\begin{equation*}
	h(t)=\sqrt{t(9.6-t)}.
\end{equation*}
The boundary condition imposed on the right side of the computational domain leads to a deformation that has a larger third component than the deformation on the left side. The resulting deformation is used as the initial deformation in part (b) of the dynamics step, which does not impose a boundary condition on this side of the domain. Starting the main gradient flow with a deformation which is higher on the right side leads to non-tent shapes, and its necessity is confirmed in practical experiments.

For the gradient flow (part (b) of the dynamics step), $\Gamma^M = \Gamma_{\text{left}}$ with prescribed value given by \eqref{eqn:BC_dynamics},  while $\Gamma^D=\emptyset$. 

The subdivision consists of 2,192 quadrilaterals and 30,267 degrees of freedom (Figure \ref{fig:flap_initial}). The initial deformation is the identity. We prescribe physical time step $\Delta t= 0.1$, pseudo-time step $\tau = 0.01$, and run the experiment for M=4 physical time steps. The stopping criterion for the gradient flow \eqref{eq:gf_original} is $tol = 10^{-6}$, and $\theta^2=10^{-3}$. Figure \ref{fig:results_flapping} in the introduction shows the equilibrium deformation at each of the four physical time steps, colored by the local isometry defect 
\begin{equation*}
	\| \nabla {\bf y}^T{\bf y} - {\bf I}_2 \|_{L^2(T)}, \quad T\in\mathcal{T}_h,
\end{equation*}
where blue represents a lower isometry defect and red represents a higher isometry defect. A high isometry defect is observed around the point where the two creases meet. This helps explain why we do not observe the same results using a bending energy with isometry constraint. Because the bending model (approximately) enforces the constraint $\sum_{T \in \mathcal{T}_h} \int_T (\nabla {\bf y}^T \nabla {\bf y} - {\bf g}) = 0$ (see \cite{Prestrain_BGNY_2022}), the gradient flow eliminates the area of high defect by ``pushing'' its location along the crease and outside the computational domain, thereby leading to the tent shape in Figure \ref{fig:preasym_flapping}.
\begin{figure}[htbp!] 
	\begin{center}
		\includegraphics[width=4.00cm]{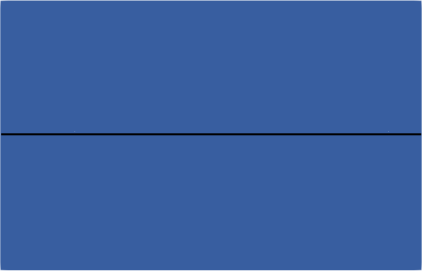}
		\hspace{1cm}
		\includegraphics[width=4.00cm]{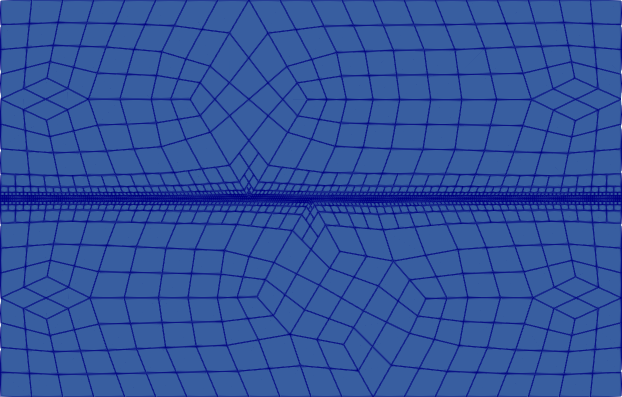}
	\end{center}
	\caption{Computational domain (left) and mesh (right) for the flapping device experiment.}
	\label{fig:flap_initial}
\end{figure}

\begin{figure}[htbp!]
	\begin{center}
		\includegraphics[width=4.00cm]{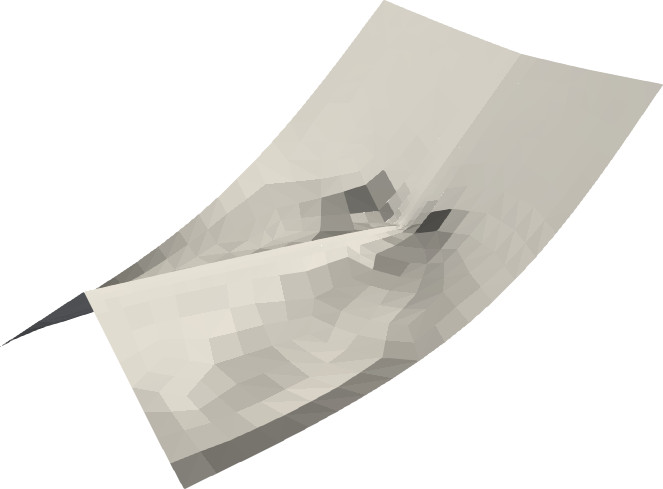}
		\hspace{1.5cm}
		\includegraphics[width=4.00cm]{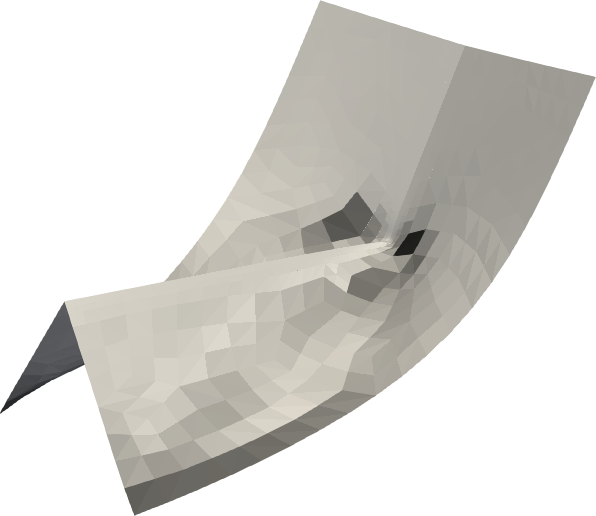} \\
		\vspace{1.3cm}
		\includegraphics[width=4.00cm]{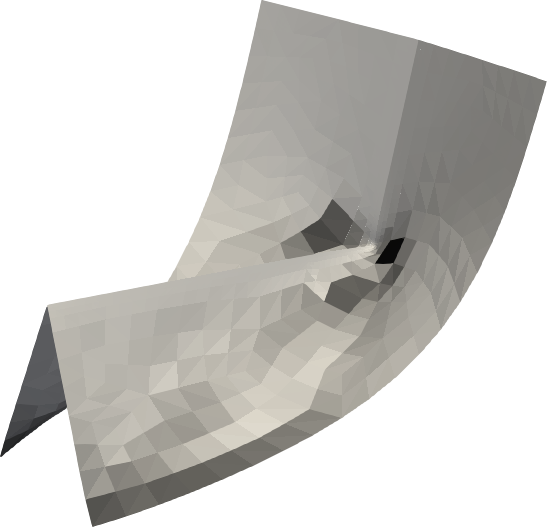}
		\hspace{1.5cm}
		\includegraphics[width=4.00cm]{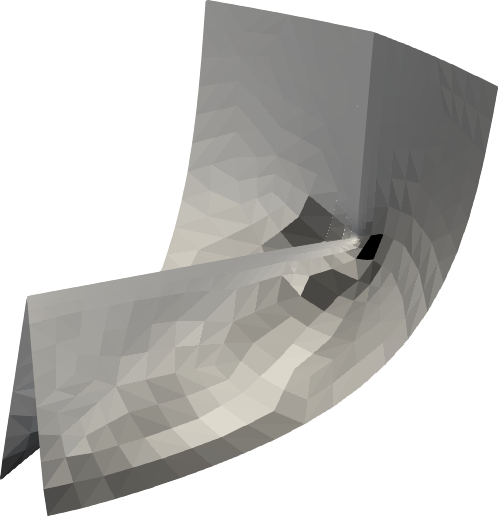}
	\end{center}
	\caption{Final configurations for the flapping device experiment, colored by prestrain defect (Black = $0.0041$, White = $0$).  Left to right and top to bottom: $m = 1, 2, 3$, and $4$.}
	\label{fig:results_flapping}
\end{figure}

\bibliography{bibliography}
\bibliographystyle{amsplain}

\end{document}